\documentclass[hidelinks,onefignum,onetabnum]{siamart220329}
\usepackage[utf8]{inputenc}
\usepackage[english]{babel}
\usepackage{microtype, fancyhdr, siunitx, graphicx, hyperref,algorithmic}
\usepackage[numbers]{natbib}
\usepackage{bm, amssymb, amsmath}
\newcommand{\I}{\bm{\mathcal{I}}} 
\newcommand{\E}{\mathbb{E}}    
\newcommand{\V}{\mathbb{V}}    
\newcommand{\Z}{\mathbb{Z}}    
\newcommand{\R}{\mathbb{R}}    
\newcommand{\bO}{\mathcal{O}}  
\newcommand{\Fr}{\mathbb{\mathcal{F}}}  
\newcommand{\Nd}{\mathcal{N}}  
\newcommand{\sv}{\, | \,}      
\newcommand{\omg}{\omega}      
\newcommand{\by}{\bm{y}}       
\newcommand{\bv}{\bm{v}}       
\newcommand{\bu}{\bm{u}}       
\newcommand{\bth}{\bm{\theta}} 
\newcommand{\bS}{\bm{\Sigma}}  
\newcommand{\sem}{\setminus}     
\newcommand{\abs}[1]{\left|#1\right|}      
\newcommand{\cvd}{\rightsquigarrow}        
\newcommand{\dif}{\mathop{}\!\mathrm{d}}   
\newcommand{\norm}[2][]{\left\Vert#2\right\Vert_{#1}} 
\newcommand{\set}[1]{\left\{ #1 \right\}}             
\newcommand{\mat}[1]{\begin{bmatrix}#1\end{bmatrix}}  
\newcommand{\tr}{\mathrm{tr}}

\DeclareMathOperator*{\argmin}{arg\,min}   

\numberwithin{equation}{section}

\newcommand{\bU}{\bm{U}}       
\newcommand{\bV}{\bm{V}}       
\newcommand{\bC}{\bm{C}}       

\newcommand{\new}[1]{\textcolor{black}{#1}}
\newcommand{\tbS}{\tilde{\bm{\Sigma}}}  
\newcommand{\tby}{\tilde{\by}}  
\newcommand{\bE}{\new{\bm{E}}} 

\title{Fast Machine-Precision Spectral Likelihoods for Stationary Time Series}
\author{Christopher J. Geoga
\thanks{Dept. of Statistics, University of Wisconsin-Madison}
}
\date{}

\begin{document}

\maketitle

\begin{abstract}
We provide in this work an algorithm for approximating a very broad class of
symmetric Toeplitz matrices to machine precision in $\mathcal{O}(n \log n)$ time
with applications to fitting time series models. In particular, for a symmetric
Toeplitz matrix $\mathbf{\Sigma}$ with values $\mathbf{\Sigma}_{j,k} = h_{|j-k|}
= \int_{-1/2}^{1/2} e^{2 \pi i |j-k| \omega} S(\omega) \mathrm{d} \omega$ where
$S(\omega)$ is piecewise smooth, we give an approximation $\mathbf{\mathcal{F}}
\mathbf{\Sigma} \mathbf{\mathcal{F}}^H \approx \mathbf{D} + \mathbf{U}
\mathbf{V}^H$, where $\mathbf{\mathcal{F}}$ is the DFT matrix, $\mathbf{D}$ is
diagonal, and the matrices $\mathbf{U}$ and $\mathbf{V}$ are in $\mathbb{C}^{n
\times r}$ with $r \ll n$.  Studying these matrices in the context of time
series, we offer a theoretical explanation of this structure and connect it to
existing spectral-domain approximation frameworks. We then give a complete
discussion of the numerical method for assembling the approximation and
demonstrate its efficiency for improving Whittle-type likelihood approximations,
including dramatic examples where a correction of rank $r = 2$ to the standard
Whittle approximation increases the accuracy of the log-likelihood
approximation from $3$ to $14$ digits for a matrix $\mathbf{\Sigma} \in
\mathbb{R}^{10^5 \times 10^5}$.  The method and analysis of this work applies
well beyond time series analysis, providing an algorithm for extremely accurate
solutions to linear systems with a wide variety of symmetric Toeplitz
matrices whose entries are generated by a piecewise smooth $S(\omega)$.
The analysis employed here largely depends on asymptotic expansions of
oscillatory integrals, and also provides a new perspective on when existing
spectral-domain approximation methods for Gaussian log-likelihoods can be
particularly problematic. 
  
  \smallskip 

\end{abstract}

\section{Introduction} \label{sec:intro}

A mean-zero, stationary, and Gaussian time series is a stochastic process
$\set{Y_t}_{t \in \Z}$ such that any contiguous collection of measurements $\by
= [Y_{t_0}, Y_{t_0 +1}, ..., Y_{t_0 + n - 1}]$ with arbitrary $t_0 \in \Z$ is
distributed as $\by \sim \Nd(\bm{0}, \bS)$, where $\bS_{j,k} = h_{\abs{j-k}}$ is
the autocovariance matrix generated by the autocovariance sequence $\set{h_k}_{k
\in \Z}$. A very special property of such covariance matrices for stationary
time series is that they are symmetric \emph{Toeplitz} matrices, meaning that
they are constant along sub- and super-diagonal bands. A fundamental and
important problem in time series analysis is fitting a parametric model to data,
which in this setting means using some parametric family of models for the
autocovariance sequence $\set{h_k(\bth)}_{k \in \Z}$ that specify the
distribution $\by \sim \Nd(\bm{0}, \bS_{\bth})$ (like, for example, (AR)(I)(MA)
process models). The canonical estimator for parameters $\bth \in \bm{\Theta}$
from data is the maximum likelihood estimator (MLE), which is computed by
minimizing the negative log-likelihood:
\begin{equation*} \label{eq:nll}
  \hat{\bth}^{\text{MLE}}
  = 
  \argmin_{\bth \in \bm{\Theta} }
  \;
  \ell(\bth \sv \by)
  =
  \argmin_{\bth \in \bm{\Theta} }
  \;
  \frac{1}{2} \left(
    \log \abs{\bS_{\bth}} + \by^T \bS_{\bth}^{-1} \by
  \right),
\end{equation*}
where constant terms in the negative log-likelihood $\ell(\bth \sv \by)$ are suppressed.

It is an elementary fact that autocovariance sequences are positive
(semi-)definite, and by Herglotz's theorem one has that
\begin{equation*} 
  \new{\bS_{j,k} = h_{|j-k|}} = \int_{[-1/2, 1/2)} e^{2 \pi i |j-k| \omg} \dif F(\omg),
\end{equation*}
where $F$ is the \emph{spectral distribution function} of $\set{Y_t}$
\cite{brockwell1991}. If $F$ has a density with respect to the Lebesgue measure
so that $\dif F(\omg) = S(\omg) \dif \omg$, we call $S(\omg)$ the \emph{spectral
density function} (SDF). An elementary property of spectral densities is that
they need only be non-negative, and symmetric about the origin in the case of
real-valued processes $Y_t$.  Considering that writing parametric families of
functions that are nonnegative and symmetric is much easier than writing
parametric families of sequences that are positive (semi-)definite, it comes as
no surprise that modeling $\set{h_k(\bth)}_{k \in \Z}$ in terms of the spectral
density $S_{\bth}(\omg)$ is an attractive and liberating option for
practitioners. 

Unfortunately, there is no particularly easy way to reformulate $\ell(\bth \sv
\by)$ directly in terms of $S(\omg)$, meaning that computing the true MLE
$\hat{\bth}^{\text{MLE}}$ requires either knowing the closed-form expression for
$\bS_{j,k} = h_{|j-k|} = \int_{-1/2}^{1/2} e^{2 \pi i \omg |j-k|} S(\omg) \dif
\omg$ for all lags $|j-k|$ or directly computing all the integrals numerically,
\new{and in both cases, for direct likelihood evaluation, one must in general build
$\bS$ entry-wise and evaluate $\ell(\bth \sv \by)$ using dense linear algebra}.  As we
will discuss, there are several popular likelihood approximations for $\by$ in
terms of $S_{\bth}(\omg)$ that are convenient and computationally expedient,
often running in $\bO(n \log n)$ complexity compared to the $\bO(n^3)$ of a
direct likelihood approximation (another fundamental challenge that motivates
many spectral approximations).  The most classical and simple of these
approximations to $\ell(\bth \sv \by)$, denoted the \emph{Whittle approximation}
\cite{whittle1953}, unfortunately yields very biased estimators.  Improvements
to this method can mitigate the challenging new sources of bias, but still come
at the cost of additional variability compared to $\hat{\bth}^{\text{MLE}}$. In
this work, we introduce a new approximation to $\ell(\bth \sv \by)$ that retains
the $\bO(n \log n)$ complexity but can be made effectively exact to double
precision with $14-15$ correct digits by exploiting piecewise smoothness of
spectral densities and special rank structure in the matrix $\Fr \bS \Fr^H$,
where $\Fr$ is the discrete Fourier transform (DFT) matrix. As a result of this
design, we obtain effectively exact evaluations of the log-likelihood $\ell(\bth
\sv \by)$ and its derivatives, as well as the ability to perform uncertainty
quantification on estimates.

\subsection{Whittle approximations and existing work} 
Let 
$$
J_n(\omg) = \frac{1}{\sqrt{n}} \sum_{j=0}^{n-1} e^{-2 \pi i j \omg} Y_j
$$
denote the discrete Fourier transform of a time series $\new{\by = }\set{Y_0,
..., Y_{n-1}}$ at frequency $\omg$. For the entirety of this work, we will
restrict focus to $\omg$ at Fourier frequencies $\set{\omg_j =
\tfrac{(j-n/2-1)}{n}}_{j=1}^n$ so that one can compute $[J_n(\omg_1), ...,
J_n(\omg_n)]$ in $\bO(n \log n)$ time with $\Fr \by$ using an FFT, where $\Fr$
is parameterized in the unitary form with the ``fftshift" for convenience as
$\Fr = [n^{-1/2} e^{-2 \pi i j k/n} ]_{j \in (-n/2):(n/2-1), k \in 0:(n-1)}$.
Motivated by the elementary observations that $\E |J_n(\omg_k)|^2 \rightarrow_n
S(\omg_k)$ and $\text{Cov}(J_n(\omg_k), J_n(\omg_{k'})) \rightarrow_n 0$ (see
\cite{brockwell1991} for an introduction), \new{one can substitute the
finite-sample (co)variances with these asymptotic ones to obtain the Whittle
approximation}
\begin{equation} \label{eq:whittle_raw}
  2 \ell^{W}(\bth \sv \by) = \sum_{j=1}^{n} \log S_{\bth}(\omg_j) +
\frac{\abs{J_n(\omg_j)}^2}{S_{\bth}(\omg_j)},
\end{equation}
\new{which has this simplified form since the asymptotic covariance matrix for
$[J_n(\omg_1), ..., J_n(\omg_n)]$ is diagonal}. Optimizing this approximation
instead of $\ell(\bth \sv \by)$ gives the estimator $\hat{\bth}^W$. Another
common way to think about this approximation is to observe that it effectively
approximates $\bS_{\bth}$ with the \emph{circulant} matrix $\bm{C}$ whose first
column $\bm{c}$ has $(\Fr \bm{c})_j = S(\omg_j)$.  This is unfortunately
incorrect, as $\bS$ is almost never exactly circulant. Further, even
if it were, it wouldn't in general for any finite $n$ be \emph{that} circulant
matrix. As a result of these finite-sample applications of limiting identities,
this estimator can exhibit severe bias for finite sample sizes and certain
varieties of spectral densities $S(\omg)$---see \cite{subba2021} for a
particularly thoughtful and concrete analysis of the sources of bias of this
approximation in terms of boundary effects. Yet another way to understand the source
of bias with this approximation is to observe that even if $\by \sim \Nd(\bm{0},
\bS_{\bth_0})$ with $[\bS_{\bth_0}]_{j,k} = \int_{-1/2}^{1/2} e^{2 \pi i \omg
|j-k|} S_{\bth_0}(\omg) \dif \omg$, $S_{\bth_0}(\omg_k)$ is \emph{not} the
variance of $J_n(\omg_k)$ for any finite $n$. And while ignoring covariance
structure between values (which the Whittle approximation also does) can in some cases
only come at the cost of efficiency, using incorrect variances will almost
always come at the cost of bias. The exact (co)variances of these DFT
coefficients is a classical computation, but because they will be used
extensively in this work we provide a terse statement of their derivation here.
\begin{proposition}
  Let $\set{Y_t}_{t=0}^{n-1}$ be a stationary mean-zero time series with spectral
  density $S(\omg)$ and $J_n(\omg_k)$ defined as above, and define
  $S_n(\omg_k, \omg_{k'}) = \textrm{Cov}(J_n(\omg_k), J_n(\omg_{k'}))$.
  Then
  \begin{equation} \label{eq:dftcov}
    S_n(\omg_k, \omg_{k'})
    =
    \frac{e^{i \frac{n-1}{n} \pi (k - k')}}{n} 
    \int_{-1/2}^{1/2} 
    D^s_n(\omg_k - \omg)
    D^s_n(\omg_{k'} - \omg)
    S(\omg) \dif \omg,
  \end{equation}
  where $D^s_n(\omg) = \frac{\sin (\pi n \omg)}{\sin (\pi \omg)} = e^{-i (n-1)
  \omg/2} \sum_{j=0}^{n-1} e^{2 \pi i j \omg}$ is a ``shifted" Dirichlet kernel.
\end{proposition}
\begin{proof}
Using Herglotz's theorem, an elementary computation shows that
\begin{align*} 
&\text{Cov}(J_n(\omg_k), J_n(\omg_{k'})) 
= 
\int_{-1/2}^{1/2} 
\left( \sum_{j=0}^{n-1} n^{-1/2} e^{2 \pi  i j (\omg - \omg_k)}   \right)
\left( \sum_{j=0}^{n-1} n^{-1/2} e^{2 \pi i j (\omg_{k'} - \omg)} \right)
S(\omg) \dif \omg
\\
&=
\frac{1}{n}
\int_{-1/2}^{1/2} 
\left( 
e^{-i (n-1) \pi \omg_k}    
\frac{\sin(\pi n (\omg_k - \omg))}{\sin(\pi (\omg_k - \omg))}
\right)
\left( 
e^{i (n-1) \pi \omg_{k'}} 
\frac{\sin(\pi n (\omg_{k'} - \omg))}{\sin(\pi (\omg_{k'} - \omg))}
\right)
S(\omg) \dif \omg.
\end{align*}
Applying the definition of $D^s_n(\omg)$, collecting terms, and bringing the
complex exponential prefactor on the trigonometric terms outside the integral
gives the results.
\end{proof}

A particularly elegant proposition for dealing with the bias of $\hat{\bth}^W$
is given in \cite{sykulski2019}: instead of using $S_{\bth}(\omg_j)$ in
$\ell^W$, use \new{the correct finite-sample marginal variances} $S_n(\omg_j,
\omg_j)$ in their place. In doing so, one immediately obtains a new likelihood
approximation whose gradient gives unbiased estimating equations (UEEs) for
parameters $\bm{\theta}$ \cite{sykulski2019,heyde1997}. This estimator, called
the \emph{debiased} Whittle estimator, is given as
\begin{equation*} 
  \hat{\bth}^{DW} = 
  \argmin_{\bth \in \bm{\Theta}} 
  \ell^{DW}(\bth \sv \by)
  =
  \argmin_{\bth \in \bm{\Theta}} 
  \frac{1}{2} \left( \sum_{j=1}^{n} \log S_{n,\bth}(\omg_j, \omg_j) +
  \frac{\abs{J_n(\omg_j)}^2}{S_{n, \bth}(\omg_j, \omg_j)} \right),
\end{equation*}
\new{Where the additional subscript of $\bth$ in $S_{n, \bth}$ is included
simply to reinforce the dependence of this function on (varying) parameters
$\bth$ and is independent from the subscript $n$, which is used to signify the
finite-sample correction.} $\hat{\bth}^{DW}$ now only comes at the cost of
additional variance from ignoring the correlation of $J_n(\omg_k)$ with
$J_n(\omg_{k'})$, and serves as a useful point of comparison with the estimator
proposed here, as it better isolates the additional variance of
Whittle-type estimators compared to $\hat{\bth}^{\text{MLE}}$.

While the standard Whittle approximation requires a pre-computed FFT of the data,
the debiased Whittle approximation requires one FFT per evaluation of
$\ell^{DW}(\bth \sv \by)$, since one can compute $\set{S_n(\omg_j,
\omg_j)}_{j=1}^{n}$ in the time domain using the classical identity that
\begin{equation} \label{eq:sn_td}
  S_n(\omg_j, \omg_j) = 2 \text{Re}\set{
    \sum_{k=0}^{n-1} (1 - n^{-1} h) h_k e^{- 2 \pi i \omg_j h}
  } - h_0.
\end{equation}
In \cite{sykulski2019} it is noted that this identity is useful to avoid
numerical integration. While this is true, it is also quite \emph{in}convenient
because it requires knowing the covariance function \new{$\set{h_k}$} as well as
the spectral density \new{$S(\omg)$}.  Considering that a large part of the
point of resorting to a Whittle-type approximation is to write a parametric
spectral density instead of a covariance function, this is a potentially serious
inconvenience. In Section \ref{sec:asexp}, in the process of building up to our own
approximation, we will discuss two performant and simple numerical integration
strategies to relax this requirement and also make this method more generally
available to practitioners.

\subsection{Our method}

The approximation we introduce here is a continuation of such spectrally
motivated likelihood methods. Letting $\bm{D}$ denote the
diagonal matrix with values $\set{S(\omg_j)}_{j=1}^{n}$, the fundamental
observation of this work is that for \new{the very broad class of
piecewise-smooth spectral densities with a small number of rough
points---including basically all standard parametric families, including
(AR)(I)(MA), Mat\'ern, rational quadratic, power law,
and other models---}we have that
\begin{equation} \label{eq:ftcov}
  \Fr \bS \Fr^H \approx \bm{D} + \bm{U} \bm{V}^H = \bm{D} + \bE,
\end{equation}
where $\bm{U}$ and $\bm{V}$ are in $\mathbb{C}^{n \times r}$ with $r \ll n$. The
standard Whittle approximation views that low-rank update as being exactly zero,
and in this light our method can be viewed as a ``corrected" Whittle
approximation. The key observation that makes working with this approximation
convenient is that the ``Whittle correction" matrix $\bE = \Fr \bS \Fr^H - \bm{D}$,
along with being severely rank-deficient in many settings, can be applied to a
vector in $\bO(n \log n)$ complexity via circulant embedding of Toeplitz
matrices, and so one can utilize the methods of randomized
low-rank approximation \cite{halko2011} to assemble $\bm{U}$ and $\bm{V}$
efficiently and scalably with implicit tools.
\new{Counter-examples of models where $\Fr \bS \Fr^H$ does not have this
structure are rare in practice, as most popular and natural models for spectral
densities $S(\omg)$ are smooth except at possibly a modest number of
locations. The category of models that likely will \emph{not} be amenable to
this structure, as will be discussed at length in Section
\ref{sec:lowrank_theory}, are those where $S(\omg)$ is either highly oscillatory
or has many rough points. Spectral densities like $S(\omg) = \gamma
\text{sinc}(\gamma \omg)^2$ or $S(\omg) = \sum_{j=1}^R \alpha_j e^{-\theta |\omg
- \omg_j|}$ for large $\gamma$ or $R$ respectively are good examples of models
for which the rank $r$ required for high accuracy is sufficiently large that the
method is not practical.}

\new{Outside of these pathological cases,}
despite having the same quasilinear runtime complexity, this representation can
be \emph{exact} to computer precision for very small $r$ (sometimes as small as
$r=2$). As one might expect, this does come at the cost of a more
expensive prefactor.  If the standard Whittle approximation requires one
pre-computed FFT of the data and subsequently runs at linear complexity, and the
debiased Whittle approximation requires one additional FFT per evaluation due to
computation of the terms in (\ref{eq:sn_td}), our approximation requires $\bO(3
r)$ FFTs to assemble (although a simpler implementation using $\bO(4 r)$ FFTs is
used in the computations of this work)
\footnote{A software package and scripts
for all computations done in this work is available at
\texttt{https://github.com/cgeoga/SpectralEstimators.jl}.}. 
For a spectral density that is neither particularly well- nor poorly-behaved, a
reasonable expectation for an $r$ that gives all $14$ significant digits of
$\ell(\bth \sv \by)$ is around $r \approx 100$---and so for full precision of
the likelihood, this method can require several hundred FFTs. But considering
how fast FFTs are, the prefactor cost for this method is sufficiently low that
it compares favorably with other high-accuracy alternatives. 

Another material advantage of this method over other spectral approximations is
that it can be used to compute uncertainties for estimators. A basic fact about
MLEs is that, under sufficient regularity conditions, $\bm{I}(\bth)^{-1/2}
(\hat{\bth}^{\text{MLE}} - \bth^{\text{true}}) \cvd \Nd(\bm{0},
\bm{\mathcal{I}})$, where $\bm{I}(\bth)$ is the expected Fisher information
matrix that will be introduced and discussed at length later in this work,
although in some settings it may actually be preferable to use the ``observed"
information matrix $H \ell(\bth \sv \by) \sv_{\bth = \hat{\bth}^{\text{MLE}}}$
\cite{efron1978}. 
\new{But when one is obtaining an estimator by optimizing some function
$\tilde{\ell}(\bth)$ other than the true log-likelihood, the Hessian of
$\tilde{\ell}$ does not have this interpretation.  The fundamental property an
approximation $\tilde{\ell}(\bth)$ needs to satisfy to provide a statistically
useful estimator is that it provides \emph{unbiased estimating equations}
(UEEs), so that $\E_{\bth_0} [ \nabla \tilde{\ell}(\bth \sv \by) \sv_{\bth =
\bth_0} ] = \bm{0}$ for any $\bth_0$ \cite{heyde1997}. But there are many such
$\tilde{\ell}$ functions that satisfy this property and are \emph{not} good
pointwise approximations to $\ell$. And in those cases, one certainly cannot
expect that the Hessian of $\tilde{\ell}$ at $\hat{\bth}^{\text{MLE}}$ resembles
the Hessian of $\ell$ at those values. Approximations like the debiased
Whittle-induced $\tilde{\ell}$ fall into this category. The method of this work,
however, can be used to obtain both estimators \emph{and} approximate
uncertainties via its Hessian matrix, since it provides a fully pointwise-accurate
approximation to $\ell(\bth)$ at for all $\bth$.}

\new{
There is a significant body of literature on the more general problem of
approximating covariance matrices for the purpose of estimation problems. Among
the most accurate varieties of approximations in this space is the use of
\emph{hierarchical matrices}, which employ rank-structure of matrix
sub-blocks in a way that provides high accuracy but retains good run-time
complexity. These methods have been applied to the Gaussian process problem in a
variety of ways
\cite{borm2003,ambikasaran2015,foreman2017,minden2017,litvinenko2019,geoga2020},
and unlike the tools proposed in this work are not specialized to
one-dimensional gridded measurements. Similarly, tools specific to Toeplitz
matrices that exploit rank-structure in off-diagonal blocks of $\Fr \bS \Fr^H$
have been proposed for fast direct linear solvers
\cite{martinsson2005,chandrasekaran2008}.
The downside to these methods is that the
algorithms are often quite complicated, and more domain-specific quantities in
statistical computing---such as log-determinants, derivatives with respect to
model parameters, and certain matrix-matrix products---are more involved to
compute, and are often computed to lower accuracy
\cite{anitescu2012,stein2013,minden2017,geoga2020}. Other popular methods in
this space that provide less accurate likelihood approximations in the pointwise
sense but are often easier to implement and analyze include sparse
approximations to $\bS^{-1}$
\cite{snelson2007,lindgren2011,nychka2015,sun2016,katzfuss2021}, direct low-rank
approximation methods \cite{cressie2008}, and matrix tapering \cite{furrer2006}.
The method presented here fits into this body of work as a specialized method
for stationary time series data. In exchange for that specialization, however,
it achieves the crushing accuracy of hierarchical matrix methods in a much
simpler and more tractable approximation format, and permits the fast and
accurate computation of domain-specific quantities like log-determinants
and derivatives of the Gaussian log-likelihood.
}

\subsection{Outline}

In the next several sections, we will introduce tools for implicitly assembling
and working with matrices of the form (\ref{eq:ftcov}). We start with a
discussion of using quadrature and asymptotic expansions to efficiently and
accurately work with $\bS$ using only the spectral density $S(\omg)$, and we
then provide a discussion of (i) why the perturbation term in (\ref{eq:ftcov})
should be low-rank, which also offers a new theoretical perspective for when
Whittle approximations are particularly problematic, and (ii) the details of
assembling the low-rank $\bm{U} \bm{V}^H$ entirely from fast matrix-vector
products. We then close by providing several tests and benchmarks to demonstrate
the speed and accuracy of the method.

\section{Oscillatory integration of spectral densities} \label{sec:asexp}

As discussed in the previous section, creating the low-rank approximation $\Fr
\bS \Fr^H \approx \bm{D} + \bm{U} \bm{V}^H$ requires working with $\bS$, the
Toeplitz matrix with values $\bS_{j,k} = h_{|j-k|}$ where $\set{h_k}$ is the
autocovariance sequence of the time series. Unlike in \cite{sykulski2019} where
both kernel and spectral density values are used to avoid numerical integration,
we will now discuss strategies for efficiently and accurately computing
integrals of the form
$
  h_k = \int_{-1/2}^{1/2} S(\omg) e^{2 \pi i k \omg} \dif \omg
$
for $k \in 0, ..., n-1$, so that one can create this factorization given only
the spectral density. Let us briefly review the challenges of such a task.
First, the standard method for obtaining the values $\set{h_k}$ using the
trapezoidal rule via the FFT may require a large number of nodes for even
moderate accuracy if, for example, $S(\omg)$ has rough points (like how $S(\omg)
= e^{-|\omg|}$ has no derivatives at the origin). A second common issue arises
if $S(\omg)$ is not periodic at its endpoints, which limits the convergence rate
of the trapezoidal rule to its baseline rate of $\bO(m^{-2})$ for $m$ nodes for
general non-smooth and/or non-periodic integrands. As a final complication, by
the Shannon-Nyquist sampling theorem, one requires at least $\bO(k)$ quadrature
nodes to resolve the oscillations of $e^{2 \pi i k \omg}$ on $[-1/2, 1/2]$, and
so any direct summation method to obtain $\set{h_k}_{k=0}^{n-1}$ would scale at
least as $\bO(n^2)$.

To bypass these issues, we propose the use of adaptive quadrature
for a fixed number of lags $k$, but then a transition to the use of asymptotic
expansions to $h_k$ for lags greater than some $k_0$. For integrands that are
not too oscillatory, standard adapative integration tools can be very
efficient and capable for any piecewise-smooth integrand. In this work, the
Julia lanuage package \texttt{QuadGK.jl} \cite{bezanson2017,quadgk} was used,
although many other tools would likely have worked equally well. Briefly,
Gaussian quadrature is a tool for achieving high-order accuracy in the numerical
integration of smooth functions. It gives approximations to integrals as
$
  \int_a^b f(x) \dif x \approx \sum_{j=1}^M \alpha_j f(x_j)
$
that are exact for polynomials up to order $2M - 1$, where
$\set{\alpha_j}_{j=1}^M$ are \emph{weights} and $\set{x_j}_{j=1}^M$ are
\emph{nodes}. In the case of $[a, b] = [-1, 1]$, the Legendre polynomials can be
used to obtain the weights and nodes and achieve an error rate of $\bO(M^{-2m -
1})$ for functions $f \in \mathcal{C}^{(m)}([a,b])$
\cite{gonnet2012review,trefethen2019approximation}. 

More interesting, however, is the discussion of methods for accurately and
efficiently computing $h_k$ when $k$ is large. Unlike traditional integration
tools, asymptotic expansion-type methods get \emph{more} accurate as $k$
increases. As in the introduction, the following result is not new (see
\cite{deano2017} for a comprehensive discussion), but we state it in the
specific form that is useful for this work since it will be referred to
frequently. Since the idea of the proof is also used repeatedly, we again
provide it here.
\begin{proposition} \label{thm:asexp}
  Let $S(\omg) \in \mathcal{C}^{(m)}([-1/2, 1/2])$ and $\int_{-1/2}^{1/2}
|S^{(m)}|(\omg) \dif \omg < C < \infty$. Then
\begin{equation*} 
  \int_{-1/2}^{1/2} S(\omg) e^{2 \pi i k \omg} \dif \omg
  =
  - \sum_{j=0}^{m-1} (-i 2 \pi k)^{-(j+1)} ( S^{(j)}(1/2) e^{\pi i k} - S^{(j)}(-1/2) e^{- \pi i k})
+
\bO(k^{-m - 1}).
\end{equation*}
\end{proposition}
\begin{proof}
  Simply apply integration by parts as many times as possible:
\begin{align*} 
  \int_{-1/2}^{1/2} S(\omg) e^{2 \pi i k \omg} \dif \omg
  &= 
  \frac{S(1/2) e^{\pi i k} - S(-1/2) e^{-\pi i k}}{-2 \pi i k}
  +
  (-2 \pi i k)^{-1} \int_{-1/2}^{1/2} S'(\omg) e^{2 \pi i k \omg} \dif \omg
  \\
  &\vdots
  \\
  &=
  \sum_{j=0}^{m-1} (-2 \pi i k)^{-(j+1)} ( S^{(j)}(1/2) e^{\pi i k} - S^{(j)}(-1/2) e^{- \pi i k})
  \\
  &+
  (- 2 \pi i k)^{-(m+1)}  \int_{-1/2}^{1/2} S^{(m)}(\omg) e^{2 \pi i \omg k} \dif \omg.
\end{align*}
\end{proof}
Remarkably, evaluating this expansion for \emph{any} lag $k$ only requires a few
derivatives of $S(\omg)$ at its endpoints. For this reason, with this tool
evaluating the tail of $\set{h_k}$ for high lags $k$ actually becomes the
fastest part of the domain to handle. And while the supposition of the above
theorem that $S \in \mathcal{C}^{(m)}([-1/2, 1/2])$ is obviously restrictive, a
simple observation extends this result to a much broader class of functions.
\begin{corollary}
\label{cor:asexp}
Let $S(\omg) \in \mathcal{C}^{(m)}([-1/2, 1/2] \sem \set{\omg_1^r, ..., \omg_L^r})$,
so that it is smooth except at locations $\set{\omg_1^r, ..., \omg_L^r}$,
further assume that at each $\omg_l^r$ it has at least $m$ directional
derivatives $S^{(m\pm)}(\omg_l^r)$ from both the left ($m-$) and the right
($m+$) for $l \in 1, ..., L$.  Then
\begin{align*} \label{eq:asexp_rough}
  \int_{-1/2}^{1/2} S(\omg) e^{2 \pi i k \omg} \dif \omg
  &=
  \sum_{l=1}^{L-1}
  \sum_{j=0}^{m-1} (-2 \pi i k)^{-(j+1)} 
  (
  S^{(j-)}(\omg_{l+1}^r) e^{2 \pi i k \omg_{l+1}^r} 
  - 
  S^{(j+)}(\omg_{l}^r) e^{2 \pi i k \omg_{l}^r} 
  )
\\ 
&+
\bO(k^{-m - 1}).
\end{align*}
\end{corollary}
This corollary, which is proven by simply breaking up the domain $[-1/2, 1/2]$
into segments with endpoints at each $\omg_l^r$ and applying the above argument
to each segment, now means that this asymptotic expansion approach can be used
to accurately calculate $h_k$ for very large $k$ for any spectral density that
is just \emph{piecewise} smooth. The prefactor on the error term in this setting
is natually increased compared to the setting of Proposition \ref{thm:asexp}.
Nonetheless, however, the convergence rates are on our side. For example, if $k
= 3\,000$ and one uses five derivatives, then the error term in the simple case
is given by $C \cdot (6,000 \pi)^{-6} \approx C \cdot 10^{-26}$ \new{for some
constant $C$}. For a spectral density like $S(\omg) = 10 e^{-10|\omg|}$, an
example that will be studied extensively in later sections due to being
well-behaved except at the origin, we see that $h_{3000} \approx 5 \cdot
10^{-7}$. For that particular function, then, this bound indicates that unless
$C$ is quite large, one can reasonably expect a full $14-15$ digits of accuracy.
For functions with few rough points, this evaluation method only
requires derivatives of $S(\omg)$ at a few locations, and can be evaluated in
just a few nanoseconds on a modern CPU once those have been pre-computed.
Combined with adaptive quadrature for sufficiently low lags $k$ that the
expansion is not yet accurate, this gives $\set{h_k}_{k=0}^{n-1}$ for all lags
quickly and accurately.

An alternative way to compute $\set{h_k}_{k=1}^n$ for potentially large $n$ in
$\bO(n \log n)$ time is to use nonuniform fast Fourier transforms (NUFFTs) and
Gauss-Legendre quadrature to evaluate $h_k$ at \emph{all} lags
\cite{dutt1993,barnett2019}. While the traditional FFT algorithm critically
depends on measurements being given on a regular grid and exploits the
symmetries that that creates, NUFFTs are fast algorithms to evaluate sums of
complex exponentials at potentially irregular locations and target frequencies
in quasilinear runtime. Per the discussion above, we see that computing
$\set{h_k}_{k=1}^n$ would require $M = \bO(n)$ quadrature nodes (to resolve the
highest frequencies) and naturally would need to be evaluated at $n$ target values,
which means that naive computation of those integrals would require $\bO(n^2)$
work. With the NUFFT, however, the sums can be computed as a matrix-vector
product
\begin{equation*}
h_k = \sum_{j=1}^M \alpha_j e^{2 \pi i k \omg_j^q} S(\omg_j^q), \quad k = 0, ..., n-1
\end{equation*} 
in $\bO(n \log n)$ time, where $\omg_j^q$ denotes the $j$-th quadrature node.
This method, while much more expensive than using asymptotic expansions, has the
advantage that it does not require computing any derivatives to evaluate the
sequence $\set{h_k}_{k=0}^{n-1}$. Because of this, it can be more accurate for SDFs
$S(\omg)$ with higher-order derivatives that potentially get very large or are
for some other reason numerically unstable to compute. With that said, however,
as discussed above, the accuracy of Gauss-Legendre quadrature depends heavily on
how smooth the integrand is, so we still recommend splitting the domain $[-1/2,
1/2)$ into panels such that the rough points of $S(\omg)$ are endpoints as with
the splitting of the asymptotic expansions. \new{If $S(\omg)$ is amenable to the
split asymptotic expansion method and $n$ is large, then the asymptotic
expansion approach will be significantly faster.  But if $S(\omg)$
has some properties that make working with higher-order derivatives infeasible,
this NUFFT-based method is an attractive alternative that still enables
practitioners to work \emph{exclusively} with $S(\omg)$ for evaluating the
log-likelihood and retain $\bO(n \log n)$ complexity.} For a much more detailed
discussion of numerical Fourier integration of spectral densities and the
quadrature-based tools introduced here for doing it quickly, we refer readers to
\cite{beckman2024}.

\section{Low rank structure of the Whittle correction} \label{sec:lowrank_theory}

We now turn to the question of \emph{why} one should expect the \new{Whittle
correction} matrix $ \bE = \Fr \bS \Fr^H - \bm{D}$
where $\bm{D}$ is diagonal with $\bm{D}_{j,j} = S(\omg_{j})$, to be of low
numerical rank. The primary tool for this investigation will again be asymptotic
expansions, and the crucial observation is that for every $k$ and $k'$,
$\bE_{k,k'}$ is given by an oscillatory integral with the same high frequency
of $2 \pi n$.  While the below analysis is not actually how we choose to compute
and assemble this low-rank approximation in this work, the following derivation
provides some idea for why the Whittle correction matrix $\bE$ has low-rank
structure, even when $S$ has non-smooth points or $\set{h_k}$ decays slowly. For the
duration of this section, we will assume that $S(\omg) \in
\mathcal{C}^{(m)}([-1/2, 1/2])$ for convenience, although the results here can
again be extended using directional derivatives if $S$ is not differentiable but
is smooth from the left and the right at rough points.

We begin by adding and subtracting $S(\omg_k)$ and $S(\omg_{k'})$ to the inner integrand
in (\ref{eq:dftcov}), which with minor simplification steps can be expressed as
\begin{align} \label{eq:dftcov_decomp}
  \frac{1}{n}
  &\int_{-1/2}^{1/2} 
  D_n^s(\omg_k - \omg) D_n^s(\omg_{k'} - \omg)
  S(\omg)
  \dif \omg
  \\
  \notag
  =
  \frac{1}{2n}
  &\int_{-1/2}^{1/2} 
  D_n^s(\omg_k - \omg) D_n^s(\omg_{k'} - \omg)
  (S(\omg) - S(\omg_k))
  \dif \omg
  \\
  \notag
  +
  \frac{1}{2n}
  &\int_{-1/2}^{1/2} 
  D_n^s(\omg_k - \omg) D_n^s(\omg_{k'} - \omg)
  (S(\omg) - S(\omg_{k'}))
  \dif \omg
  \\
  \notag
  +
  \frac{S(\omg_k) + S(\omg_{k'})}{2n}
  &\int_{-1/2}^{1/2}
  D_n^s(\omg_k - \omg) D_n^s(\omg_{k'} - \omg)
  \dif \omg.
\end{align}
While unwieldy, this representation of the integral already provides a
reasonably direct explanation about several features of $\Fr \bS \Fr^H$.
Recalling that
\begin{equation*} 
  \int_{-1/2}^{1/2} D_n^s(\omg_k - \omg) D_n^s(\omg_{k'} - \omg) \dif \omg
  =
  \begin{cases}
    n & k = k' 
    \\
    0 & k \neq k' 
  \end{cases},
\end{equation*}
we see that the third term is precisely the diagonal contribution of
$S(\omg_k)$. We will now argue that the first two terms in the above sum
correspond to severely rank-deficient matrices.  Since they can of course be
analyzed in the exact same way, we study only the first in detail here. 

To begin, we slightly rewrite the first term in (\ref{eq:dftcov_decomp}) as
\begin{equation} \label{eq:dftcov2_p1}
  \int_{-1/2}^{1/2}
  \sin(\pi n (\omg_k - \omg)) \sin(\pi n (\omg_{k'} - \omg))
  \csc(\pi(\omg_k - \omg)) \csc(\pi(\omg_{k'} - \omg))
  (S(\omg) - S(\omg_k)) \dif \omg.
\end{equation}
From here, we partition the domain into 
\begin{equation} \label{eq:domain_split}
  [-1/2, 1/2] 
  =
  \underbrace{[-1/2, \omg_k - \gamma]}_{\text{Type I}}
  \cup
  \underbrace{B_{\gamma}(\omg_k)}_{\text{Type II}}
  \cup
  \underbrace{[\omg_k + \gamma, \omg_{k'} - \gamma]}_{\text{Type I}}
  \cup
  \underbrace{B_{\gamma}(\omg_{k'})}_{\text{Type II}}
  \cup
  \underbrace{[\omg_{k'} + \gamma, 1/2]}_{\text{Type I}},
\end{equation}
where $B_{\gamma}(x) = [x - \gamma, x + \gamma]$ and $\gamma$ is some small
number chosen to keep distance from the singularities of the cosecant terms. By
design, then, in Type I intervals the cosecant terms are simple analytic
functions. Defining 
\begin{equation} \label{eq:St}
  \tilde{S}_{n, k, k'}(\omg) = 
  \csc(\pi(\omg_k - \omg)) \csc(\pi(\omg_{k'} - \omg)) (S(\omg) - S(\omg_k)),
\end{equation}
we see that in Type I regions this function is bounded above (assuming that $S$
itself is) and as smooth as $S$. This motivates the following result that will
be used many times in this section.
\begin{proposition}
\label{prop:dftcov_as}
If $g(\omg) \in \mathcal{C}^{(m)}([a,b])$ and $g^{(m)}$ is integrable on $[a,b]$,
then
  \begin{align} \label{eq:dftcov_as}
    &\int_a^b 
    \sin(\pi n (\omg_k - \omg)) \sin(\pi n (\omg_{k'} - \omg))
    g(\omg) 
    \dif \omg
    =
    \\
    &\frac{(-1)^{k - k'}}{2} 
      \int_a^b g(\omg) \dif \omg
    -
    \frac{1}{2} 
    \Re\set{
    e^{i \pi n \frac{k+k'}{2}} 
    \left[
      e^{2 \pi i b} \sum_{j=0}^{m-1} \frac{g^{(j)}(b)}{(-i 2 \pi n)^{j+1}}
      -
      e^{2 \pi i a} \sum_{j=0}^{m-1} \frac{g^{(j)}(a)}{(-i 2 \pi n)^{j+1}}
    \right]
    }
    \\
    \notag
    &+
    \bO(n^{-m-1}),
  \end{align}
  \new{where $\Re(a + ib) = a$ denotes the real part of a complex number.}
\end{proposition}
\begin{proof}
  Using the product-to-sum angle formula that $\sin(\theta) \sin(\phi) =
  \tfrac{1}{2}(\cos(\theta - \phi) - \cos(\theta + \phi))$ and standard
  manipulations of the complex exponential, the left hand side equation can be
  re-written as
  \begin{equation*} \label{eq:r2_2}
  \frac{1}{2} \cos \left(\pi n \frac{k - k'}{n} \right) \int_{a}^b g(\omg) \dif \omg
  +
  \frac{1}{2} \Re \set{
    e^{i \pi n \frac{k + k'}{2}}
    \int_{a}^b e^{-i 2 \pi n \omg} g(\omg) \dif \omg
  }.
  \end{equation*}
  But $\cos(\pi n (k - k')) = (-1)^{k-k'}$, which provides the simplification of
  the first term. For the second term, just as in the proof of Proposition
  \ref{thm:asexp}, we simply do integration by parts as many times as possible
  on the now standard-form oscillatory integral in the second term to get
  \begin{equation*} 
    \int_a^b e^{-2 \pi i n \omg} g(\omg) \dif \omg
    =
    - \sum_{j=0}^{m-1} (-2 \pi i n)^{-j-1} \set{
    g^{(j)}(b) e^{-2 \pi i n b}
    -
    g^{(j)}(a) e^{-2 \pi i n a}
    }
    +
    \bO(n^{-m-1}).
  \end{equation*}
  An elementary rearrangement gives the result.
\end{proof}
While this proposition is presented in generality, there is a reasonable amount
of additional simplification that one can do with the right-hand side based on
the value of $k + k' \pmod{4}$. In particular, we see that $e^{i \pi n
(k+k')/2} = i^{k-k' \pmod{4}}$, which means that every single complex
exponential on the right-hand side of (\ref{eq:dftcov_as}) simplifies nicely. If
we assume that $k = k' \pmod{4}$ so that $e^{i \pi n (k+k')/2} = 1$, the
asymptotic expansion term can be more concretely expanded to
\begin{align} \label{eq:dftcov_as_elaborate}
\notag
&\Re\set{
  (\cos(2 \pi b) + i \sin(2 \pi b))
  \sum_{j=0}^{m-1} \frac{g^{(j)}(b)}{(-i 2 \pi n)^{j+1}}
  -
  (\cos(2 \pi a) + i \sin(2 \pi a))
  \sum_{j=0}^{m-1} \frac{g^{(j)}(a)}{(-i 2 \pi n)^{j+1}}
}
\\
=
&\cos(2 \pi b) 
\sum_{j=0, \, j \text{ odd}}^{m-1} (-1)^{1+\lfloor j/2 \rfloor} 
\frac{g^{(j)}(b)}{(2 \pi n)^{j+1}}
-
\cos(2 \pi a) 
\sum_{j=0, \, j \text{ odd}}^{m-1} (-1)^{1+\lfloor j/2 \rfloor} 
\frac{g^{(j)}(a)}{(2 \pi n)^{j+1}}
\\
\notag
+
&\sin(2 \pi b) 
\sum_{j=0, \, j \text{ even}}^{m-1} (-1)^{2+\lfloor j/2 \rfloor} 
\frac{g^{(j)}(b)}{(2 \pi n)^{j+1}}
-
\sin(2 \pi a) 
\sum_{j=0, \, j \text{ even}}^{m-1} (-1)^{2+\lfloor j/2 \rfloor}
\frac{g^{(j)}(a)}{(2 \pi n)^{j+1}}.
\end{align}
And while we don't enumerate the other cases for the three values of $e^{i \pi n
(k+k')/2}$, the only difference is in the odd/even indexing and the alternating
sign inside the sum. The key takeaway here is that the entire integral
(\ref{eq:dftcov2_p1}) can actually be written as a very smooth integral term
($\int_a^b g(\omg) \dif \omg$), four trigonometric functions that are
non-oscillatory since they have unit wavelength and $a, b \in [-1/2, 1/2]$ with
coefficients taken from functions that are themselves smooth (since $g \in
\mathcal{C}^{(m)}([a,b])$), and a remainder term depending on high-order
derivatives of $g(\omg)$ on $[a, b]$.  As it pertains to the Whittle correction
matrix $\bE$, for each of the Type-I regions we have endpoints $(a = -1/2, b =
\omg_k - \gamma)$, $(a = \omg_k + \gamma, b = \omg_{k'} - \gamma)$, and $(a =
\omg_{k'}+\gamma, b = 1/2)$ respectively, and so the effective contribution of
the Type-I integrals to $\bE_{k,k'}$ is given by 
$
  \int_{[-1/2, 1/2] \sem (B_{\gamma}(\omg_k) \cup B_{\gamma}(\omg_{k'}))}
  \tilde{S}_{n,k,k'}(\omg) \dif \omg
$
and a total of $12$ trigonometric functions at the three sets of given endpoints
$a$ and $b$ with smoothly varying coefficients. Given that the trigonometric
functions are analytic and that $g \in \mathcal{C}^{(m)}([-1/2, 1/2])$, since
very smooth kernel matrices often exhibit rapid spectral decay (an observation
dating back to at least \cite{greengard1987} with the Fast Multipole Method
(FMM)),  we expect the contribution of the Type-I intervals to $\bE$ to have
exceptionally fast spectral decay and be severely rank-deficient so long as the
remainder from the asymptotic expansion is small. With slightly more effort, we
may repeat this analysis on the Type-II regions with singularities.

For the first Type II region of $[\omg_k - \gamma, \omg_k + \gamma]$, the story is
only slightly more complicated. This time we introduce the condensed notation of
\begin{equation*} 
  \tilde{S}_{n, k'}(\omg) = \csc(\pi(\omg_{k'} - \omg))(S(\omg) - S(\omg_k))
\end{equation*}
for the bounded and non-oscillatory part of the integrand, and we study
\begin{equation} \label{eq:dftcov2_typeII}
  \int_{\omg_k + \gamma}^{\omg_k - \gamma} \sin(\pi n (\omg_k - \omg)) \sin(\pi n (\omg_{k'} - \omg))
  \csc(\pi(\omg_k - \omg)) \tilde{S}_{n, k'}(\omg) \dif \omg.
\end{equation}
The important observation to make here is that the singularity presented by
$\csc(\pi (\omg_k - \omg))$ is simple. Recalling the Laurent expansion $\csc(t) =
t^{-1} + \tfrac{t}{6} + \tfrac{7t^2}{360} + ...$, we see that one can simply
``subtract off" the singularity and obtain a standard power-series type
representation of $\csc(t) - t^{-1} \approx P_l(t)$ for $t \in [-\gamma,
\gamma]$ and some low-order polynomial $P_l$ based on the Laurent series.  This
motivates the decomposition of (\ref{eq:dftcov2_typeII}) into
\begin{align*} \label{eq:dftcov2_typeII_decomp}
  &\int_{\omg_k + \gamma}^{\omg_k - \gamma} \sin(\pi n (\omg_k - \omg)) \sin(\pi n (\omg_{k'}
  - \omg)) P_l(\omg_k - \omg) \tilde{S}_{n, k'}(\omg) \dif \omg
  \\
  \notag
  +
  &\int_{\omg_k + \gamma}^{\omg_k - \gamma} \sin(\pi n (\omg_k - \omg)) \sin(\pi n (\omg_{k'}
  - \omg)) \frac{\tilde{S}_{n, k'}(\omg)}{\omg_k - \omg} \dif \omg.
\end{align*}
The first term above can again be expanded as a small number of unit-wavelength
trigonometric functions via Proposition \ref{prop:dftcov_as} with $g(\omg) =
P_l(\omg_k - \omg) \tilde{S}_{n,k'}(\omg)$. Because $\tilde{S}_{n,k'}(\omg) = 0$ at
$\omg = \omg_k$ and has been assumed to be smooth, the second term is also not
singular and is a standard oscillatory integral. Another application of
Proposition \ref{prop:dftcov_as} with $g(\omg) = (\omg_k - \omg)^{-1}
\tilde{S}_{n,k'}(\omg)$ can be applied, making the entire contribution of
(\ref{eq:dftcov2_typeII}) expressible as a linear combination of a small number
of unit-wavelength trigonometric functions with smoothly varying coefficients.

The final segment of $[-1/2, 1/2]$ to study is the Type II region $[\omg_{k'} -
\gamma, \omg_{k'} + \gamma]$, where the singularity caused by $\csc(\pi
(\omg_{k'} - \omg))$ has not already been factored out. Introducing one final
condensed integrand notation of
\begin{equation*} 
  \tilde{S}_{n, k}(\omg) = \csc(\pi(\omg_k - \omg))(S(\omg) - S(\omg_k)),
\end{equation*}
we again decompose the contribution of that segment as
\begin{align*} 
  &\int_{\omg_{k'} + \gamma}^{\omg_{k'} - \gamma} \sin(\pi n (\omg_k - \omg)) \sin(\pi n (\omg_{k'} - \omg))
  \csc(\pi(\omg_{k'} - \omg)) \tilde{S}_{n, k}(\omg) \dif \omg
  \\
  =
  \notag
  &\int_{\omg_{k'} + \gamma}^{\omg_{k'} - \gamma} \sin(\pi n (\omg_k - \omg)) \sin(\pi n
  (\omg_{k'} - \omg)) P_l(\omg_{k'} - \omg) \tilde{S}_{n, k}(\omg) \dif \omg
  \\
  \notag
  +
  &\int_{\omg_{k'} + \gamma}^{\omg_{k'} - \gamma} \sin(\pi n (\omg_k - \omg)) \sin(\pi n
  (\omg_{k'} - \omg)) \frac{\tilde{S}_{n, k}(\omg)}{\omg_{k'} - \omg} \dif \omg.
\end{align*} 
The first term in the divided integral can be handled yet again by Proposition
\ref{prop:dftcov_as} just as above. The only new wrinkle is that the second term
in the right-hand side now still has a singularity. \new{Thankfully, with one more
integral splitting we can re-express this last term with}
\begin{align*} 
  &\int_{\omg_{k'} + \gamma}^{\omg_{k'} - \gamma} \sin(\pi n (\omg_k - \omg)) \sin(\pi n
  (\omg_{k'} - \omg)) \frac{\tilde{S}_{n, k}(\omg)}{\omg_{k'} - \omg} \dif \omg
  \\
  =
  &\int_{\omg_{k'} + \gamma}^{\omg_{k'} - \gamma} \sin(\pi n (\omg_k - \omg)) \sin(\pi n
  (\omg_{k'} - \omg)) \frac{\tilde{S}_{n, k}(\omg) - \tilde{S}_{n,k}(\omg_{k'})}{\omg_{k'} - \omg} \dif \omg
  \\
  +
  \tilde{S}_{n,k}(\omg_{k'}) 
  &\int_{\omg_{k'} + \gamma}^{\omg_{k'} - \gamma} \sin(\pi n (\omg_k - \omg)) 
  \sin(\pi n (\omg_{k'} - \omg)) \frac{1}{\omg_{k'} - \omg} \dif \omg.
\end{align*}
\new{The first term on the right-hand side is no longer singular because
$\tilde{S}_{n,k}(\omg) - \tilde{S}_{n,k}(\omg_{k'})$ is smooth and goes to zero
at $\omg = \omg_{k'}$ faster than $(\omg_{k'} - \omg)^{-1}$ blows up, and so one
may apply Proposition \ref{prop:dftcov_as}. The second integral is the integral
of $\sin$ against a bounded function since $\sin(\pi n (\omg_{k'} - \omg_k))$ is
also smooth and goes to zero as $\omg \rightarrow \omg_{k'}$, and thus can
be handled with the even more simple asymptotic expansion analagous to the one
used in Proposition \ref{thm:asexp} that again represents it as a small number
of smooth functions.}

Combining all of these segmented contributions, we finally conclude in
representing $\bE_{k,k'}$ as a small number of very smooth functions. And while
ostensibly breaking the integral into five pieces, two of which need to be
broken up again into two or three additional terms to be suitable for expansion,
should lead to so many smooth terms that the matrix is not particularly
rank-deficient, we note that all of these trigonometric functions have unit
wavelength and are potentially linearly dependent with each other, and that
derivatives of the corresponding function $g$ in each application of Proposition
\ref{prop:dftcov_as} are by supposition also well-behaved. 

This derivation is also illuminating about the sources of error in Whittle
approximations and the degree of rank-structure in the Whittle correction
matrices. For example, if $S$ has rough points where it is continuous but not
smooth, then the above domain partitioning will have to be refined so that
Corollary $1$ can be applied to fix the accuracy of the asymptotic
approximation. This will then mean that $\bE_{k,k'}$ will need more asymptotic
expansion-sourced terms for the expansion to be accurate. \new{For spectral
densities $S(\omg)$ that have many such rough points, the number of unique terms
in the split expansion can certainly add up to a degree that ruins this rank
structure. And if $S(\omg)$ and its derivatives are highly oscillatory, then
terms like (\ref{eq:dftcov_as_elaborate}) may not actually yield particularly
low-rank matrices despite being smooth.}

\section{Numerical assembly and action of the Whittle correction} \label{sec:lowrank_num}
With this theoretical justification for the low-rank structure of the Whittle
correction matrix, we now turn to discussing how to actually compute and
assemble the matrices $\bm{U}$ and $\bm{V}$ such that $\Fr \bS \Fr^H - \bm{D}
\approx \bm{U} \bm{V}^H$. First, we note that the action of $\bS$ onto vectors
can be done in $\bO(n \log n)$ with the familiar circulant embedding technique
\cite{wood1994} once the values $\set{h_k}_{k=0}^{n-1}$ have been computed.
While we refer readers to the references for a more complete discussion of
circulant embedding, we will provide a brief discussion here. The symmetric
Toeplitz matrix $\bS$ is specified entirely by its first column $[h_0, h_1, ...,
h_{n-1}]$. This column can be extended to one of length $2n - 1$ given by
$\bm{c} = [h_0, h_1, ..., h_{n-1}, h_{n-1}, h_{n-2}, ..., h_1]$, and if one
assembles a Toeplitz matrix with this length $2n - 1$ column, the matrix is in
fact circulant and diagonalized by the FFT \cite{gray2006}. Since
the FFT can be applied in $\bO(n \log n)$ complexity, so can this
augmented circulant matrix. Thus, one can compute $\bS \bv$ for any vector $\bv$
by extracting the first $n$ components of $\bm{C} [\bv, \bm{0}_{n-1}]$, where
$\bm{C}$ is the circulant matrix made with $\bm{c}$, and $\bm{0}_{n-1}$ is $n-1$
many zeros appended to $\bv$.

Letting $\Fr_n$ be the FFT of size $n$ (briefly now defined without the
``fftshift"), we can thus summarize the action of $\Fr_n \bS \Fr_n^H - \bm{D}$
on a vector $\bv$ with
\begin{equation} \label{eq:fast_sv}
  (\Fr_n \bS \Fr_n^H - \bm{D}) \bv
  =
  \Fr_n \left\{
    \Fr_{2n-1}^H
    \left(
    \set{\Fr_{2n-1} \bm{c}}
    \circ
    \Fr_{2n-1} \mat{ \Fr_{n}^H \bv \\ \bm{0}_{n-1}}
    \right)
  \right\}_{1:n}
  -
  \bm{D}\bv.
\end{equation}
Assuming that $\Fr_{2n-1} \bm{c}$ is pre-computed, this means that the action of
$\Fr \bS \Fr^H - \bm{D}$ on $\bv$ requires four FFTs, two of size $n$ and two of
size $2n-1$. 

With this established, we turn to the process of approximating $\bC$ from the
position of only being able to apply it to vectors efficiently.  The field of
randomized low-rank approximations has become an important part of modern
numerical linear algebra, and we refer readers to review papers like
\cite{halko2011} and the more recent \cite{tropp2023} for broad introductions
and context. As it pertains to this work, it is sufficient to discuss the
problem of using randomized algorithms to estimate the range of a matrix.
Letting $\bm{A} \in \R^{n \times n}$ denote an arbitrary matrix of rank $r$,
\cite{halko2011} shows that the orthogonal matrix $\bm{Q}$ such that
\begin{equation*} 
  \bm{A} \bm{\Omega} = \bm{Q} \bm{R},
\end{equation*}
where $\bm{\Omega} \in \R^{n \times (r+p)}$ is a ``sketching"/test matrix, which
in this work will be made of i.i.d. standard normal random variables (although
other and potentially faster choices are available), is an approximate basis for
the column space of $\bm{A}$. \new{The matrix $\bm{Q}$ itself may be obtained by
performing a standard QR factorization on the product $\bm{A} \bm{\Omega}$}, and
the parameter $p$ is an oversampling parameter which is often picked to be some
small number like $p=5$ \cite{halko2011}. 
From there, one obtains a simple low-rank representation of
$\bm{A}$ as
\begin{equation*} 
  \bm{A} \approx \underbrace{\bm{Q}}_{\bm{U}} \underbrace{\bm{Q}^H
  \bm{A}}_{\bm{V}^H}.
\end{equation*}
\new{This entire process is summarized in Algorithm \ref{alg:assemble}.}

\begin{algorithm}
\caption{\new{Compute the approximation $\Fr \bS \Fr^H \approx \bm{D} + \bm{U}
\bm{V}^H$}}
\label{alg:assemble}
\begin{algorithmic}
  \STATE{\new{1. Using a method from Section \ref{sec:asexp}, obtain the
  autocovariance terms $\set{h_j}_{j=0}^{n-1}$}.}
  \STATE{\new{2. Create diagonal matrix $\bm{D}_{j,j} = S(\omg_j)$ at Fourier
  frequencies $\set{\omg_j}_{j=1}^n$}.}
  \STATE{\new{3. Draw sketching matrix $\bm{\Omega} \in \R^{n \times (p+r)}$ with
  $\bm{\Omega}_{j,k} \overset{\text{i.i.d.}}{\sim} \Nd(0,1)$}.}
  \STATE{\new{4. Using Equation \ref{eq:fast_sv}, compute $(\Fr_n \bS \Fr_n^H - \bm{D})
  \bm{\Omega}$}.}
  \STATE{\new{5. Compute $(\Fr_n \bS \Fr_n^H - \bm{D}) \bm{\Omega} = \bm{Q}
  \bm{R}$}.}
  \STATE{\new{6. Set $\bm{U} = \bm{Q}$}.}
  \STATE{\new{7. Using Equation \ref{eq:fast_sv} again, compute $\bm{V} = (\Fr_n \bS
  \Fr_n^H - \bm{D}) \bm{Q}$}.}
\end{algorithmic}
\end{algorithm}

This low-rank representation can easily be converted to other truncated
factorization types like a partial SVD or eigendecomposition \cite{halko2011}.
\new{Since it will be used in evaluating the log-likelihood via the
Sherman-Morrison-Woodbury inversion formula, however, in the case where one
chooses to convert to a partial eigendecomposition or SVD but overestimates the
numerical rank in their choice of $r$, it is prudent for the sake of
conditioning to discard degenerate eigen- or singular vectors.} Connecting
this to Section \ref{sec:asexp}, then, we see that obtaining a rank $r$
approximation $\Fr \bS \Fr^H - \bm{D}$ requires applying it to $r + p$ many
random vectors. While there is sufficient structure in this specific
matrix-vector application that one could reduce the number of FFTs from
$4(r+p)$, the implementation used in this work does not employ any particular
optimization of that form and is nonetheless satisfyingly fast for an
effectively exact method. 

\subsection{Approximating the log-likelihood} 

Armed with this low-rank approximation, one now has all of the pieces for
working with $\Fr \bS \Fr^H \approx \bm{D} + \bm{U} \bm{V}^H$.
Letting 
\begin{equation} \label{eq:appx_nll}
  2 \tilde{\ell}(\bth \sv \Fr \by)
  =
  \log \abs{\bm{D} + \bm{U} \bm{V}^H} + (\Fr \by)^H (\bm{D} + \bm{U}
  \bm{V}^H)^{-1} (\Fr \by)
  \approx 
  2 \ell(\bth \sv \by)
\end{equation}
denote the approximated log-likelihood, we see that, post-assembly of $\bm{U}$
and $\bm{V}$, evaluation of $\tilde{\ell}(\bth \sv \bv)$ for an arbitrary vector
$\bv$ is possible in linear complexity. For the log-determinant term, using the
matrix determinant lemma we see that
\begin{equation*} 
  \log \abs{\bm{D} + \bm{U} \bm{V}^H}
  =
  \log\abs{\I_r + \bm{V}^H \bm{D}^{-1} \bm{U}} + \log \abs{\bm{D}}.
\end{equation*}
For the quadratic form, we may compute it using the Sherman-Morrison-Woodbury
formula as
\begin{equation*} 
\bv^T (\bm{D} + \bm{U} \bm{V}^H)^{-1} \bv
=
  \bv^T \bm{D}^{-1} \bv
  -
  \bv^T \bm{D}^{-1} \bm{U}(\I_r + \bm{V}^H \bm{D}^{-1} \bm{U})^{-1} \bm{V}^H
  \bm{D}^{-1} \bv.
\end{equation*}
Since all of the matrix-matrix operations in the right-hand sides of those two
equations are for small $\R^{r \times r}$ matrices, we see that both the
log-determinant and the quadratic form can be evaluted in $\bO(n)$ complexity.
And since the assembly of $\bm{D} + \bm{U} \bm{V}^H$ is done in $\bO(n \log n)$,
we conclude that the full job of assembling the approximation and then
evaluating the approximated log-likelihood runs in quasilinear time complexity and
linear storage complexity.

\subsection{Gradients, factorization, and information matrices} 

If one assumes that a parametrically indexed spectral density $S_{\bth}(\omg)$
has partial derivatives $\partial_{\theta_j} S_{\bth}(\omg)$ that share the same
smoothness structure as $S_{\bth}$ with respect to $\omg$ (such as being
piecewise $\mathcal{C}^{(m)}([-1/2, 1/2])$), then the above arguments about the
low-rank structure of $\Fr \bS \Fr^H$ apply equally well to $\Fr \set{
\partial_{\theta_j} \bS_{\bth} } \Fr^H$. For notational clarity, throughout this
section we absorb the DFT matrix $\Fr$ into $\bS$ and write all equations
\new{in terms of the matrix $\tbS_{\bth} = \Fr \bS_{\bth} \Fr^H$ and transformed
data vector $\tilde{\by} = \Fr \by$}. 

To start, we introduce the subscript $j$ and let $\bm{D}_j$, $\bm{U}_j$, and
$\bm{V}_j$ be matrices such that
\begin{equation} \label{eq:lrdj}
  \frac{\partial}{\partial \theta_j} \tbS_{\bth}
  \approx
  \bm{D}_j + \bm{U}_j \bm{V}_j^H.
\end{equation}
The $j$-th term of the gradient $\nabla_{\bth} \tilde{\ell}(\bth \sv \tby)$ is
given by
\begin{equation} \label{eq:grad}
  2 [\nabla \tilde{\ell}(\bth \sv \tby)]_j
  =
  \text{tr}\left(
  \tbS_{\bth}^{-1} 
  \set{
    \frac{\partial}{\partial \theta_j} \tbS_{\bth}
  }
  \right)
  -
  \tby^H
  \tbS_{\bth}^{-1} \set{
    \frac{\partial}{\partial \theta_j} \tbS_{\bth}
  }
  \tbS_{\bth}^{-1}
  \tby.
\end{equation}
Observing this equivalent structure in (\ref{eq:lrdj}), we note that the
quadratic form term in (\ref{eq:grad}) can easily be evaluated in $\bO(n)$
complexity once the two low-rank approximations have been assembled by again
making use of the Sherman-Morrison-Woodbury formula and the speed of simple
matrix-vector products. More interesting, however, and what sets this
approximation method apart from more complex general-purpose matrix compression
methods is that the trace can also be computed exactly without resorting to
stochastic trace estimation (see \cite{geoga2020} and references therein for
discussion). In particular, note that the Sherman-Morrison-Woodbury formula
given in the above section implies that $(\bm{D} + \bm{U} \bm{V}^H)^{-1}$ can be
again represented as a low rank perturbation of a diagonal matrix, which we will
denote as $\tilde{\bm{D}} + \tilde{\bm{U}} \tilde{\bm{C}} \tilde{\bm{V}}^H$.
Applying this, we see that
\begin{align} \label{eq:dprod}
  \notag
  (\bm{D} + \bm{U} \bm{V}^H)^{-1}(\bm{D}_j + \bm{U}_j \bm{V}_j^H)
  &=
  (\tilde{\bm{D}} + \tilde{\bm{U}} \tilde{\bm{C}} \tilde{\bm{V}}^H)
  (\bm{D}_j + \bm{U}_j \bm{V}_j^H)
  \\
  &=
  \notag
  \tilde{\bm{D}} \bm{D}_j
  +
  \tilde{\bm{D}} \bm{U}_j \bm{V}^H
  +
  \tilde{\bm{U}} \tilde{\bm{C}} \tilde{\bm{V}}_H \bm{D}_j
  +
  \tilde{\bm{U}} \tilde{\bm{C}} \tilde{\bm{V}}^H
  \bm{D}_j + \bm{U}_j \bm{V}_j^H
  \\
  &=
  \tilde{\bm{D}} \bm{D}_j
  +
  \bm{M}.
\end{align}
While this looks problematic, we note that the matrix $\bm{M}$ defined as the
last three terms in the above equation is a sum of three matrices whose rank is
$\leq r$, and so the rank of $\bm{M}$ is at most $3 r$. Second, because each of
those terms is already a low-rank representation that can be applied to vectors
in $\bO(n)$ time, we can simply repeat the randomized low-rank approximation
strategy from earlier, computing first a randomized basis for the column space
with 
\begin{equation} \label{eq:dprod_sketch}
  (\tilde{\bm{D}} \bm{U}_j \bm{V}^H
  +
  \tilde{\bm{U}} \tilde{\bm{C}} \tilde{\bm{V}}^H \bm{D}_j
  +
  \tilde{\bm{U}} \tilde{\bm{C}} \tilde{\bm{V}}^H
  \bm{D}_j + \bm{U}_j \bm{V}_j^H)
  \bm{\Omega}
  =
  \bm{Q}_j \bm{R}_j
\end{equation}
and then re-compressing to a single low-rank representation with $\bm{Q}_j
\bm{Q}_j^H \bm{M}$. From there, we observe the simple fact that
$\text{tr}(\bm{A} + \bm{B} \bm{C}^T) = \text{tr}(\bm{A}) + \text{tr}(\bm{C}^T
\bm{B})$ for any matrices $\bm{A}$, $\bm{B}$, and $\bm{C}$ and conclude that the
trace in (\ref{eq:grad}) can be computed exactly and in $\bO(n)$ time. \new{This
procedure for computing the gradient in quasilinear complexity is summarized in
Algorithm \ref{alg:grad}}.  As will be demonstrated in the next section, these
gradients are computed to similarly high accuracy as the log-likelihood
evaluation itself, providing effectively the same number of digits.

\begin{algorithm}
\caption{\new{Compute the gradient approximation (\ref{eq:grad})}}
\label{alg:grad}
\begin{algorithmic}
  \WHILE{$j < p$}
  \STATE{\new{1. Using Algorithm \ref{alg:assemble} with $S_{\bth}(\omg)$, create $\bm{D} + \bm{U}
  \bm{V^T}$}.}
  \STATE{\new{2. Using Algorithm \ref{alg:assemble} with
  $\frac{\partial}{\partial \theta_j} S_{\bth}(\omg)$, assemble 
  $\bm{D}_j + \bm{U}_j \bm{V}_j^T$ from (\ref{eq:lrdj})}.}
  \STATE{\new{3. Using (\ref{eq:dprod}), compute $\bm{M}$}.}
  \STATE{\new{4. Draw sketching matrix $\bm{\Omega} \in \R^{n \times (3p + r)}$
  with $\bm{\Omega}_{j,k} \overset{\text{i.i.d.}}{\sim} \Nd(0,1)$}.}
  \STATE{\new{5. Re-compress $\bm{Q}_j \bm{Q}_j^H \bm{M} = \bm{A}
  \bm{B}^T$ with the column space basis $\bm{Q}_j$ of (\ref{eq:dprod_sketch})}.}
  \STATE{\new{6. Compute $\tr(\tbS_{\bth}^{-1}
  \set{\frac{\partial}{\partial \theta_j} \tbS_{\bth}}) = \tr(\tilde{\bm{D}} \bm{D}_j)
  + \tr(\bm{B}^T \bm{A})$}.}
  \STATE{\new{7. Compute 
  $
  \by^H
  \tbS_{\bth}^{-1} \set{
    \frac{\partial}{\partial \theta_j} \tbS_{\bth}
  }
  \tbS_{\bth}^{-1}
  \by
  $
  using the above representations and the Sherman-Morrison-Woodbury identity}.}
  \STATE{\new{8. Obtain the gradient term (\ref{eq:grad})}.}
  \ENDWHILE
\end{algorithmic}
\end{algorithm}

Another nice property of this approximation structure for $\tbS$ is that it
admits a simple and direct symmetric factorization. In the above discussions for
notational simplicity we have worked with the low rank form $\bm{D} + \bm{U}
\bm{V}^H$. But as previously mentioned, these low-rank representations can be
easily and quickly converted to other structures like a partial SVD or
eigendecomposition \cite{halko2011}. For this factorization, it is easiest to
work with a representation like $\bm{D} + \bm{U} \bm{\Lambda} \bm{U}^H$, where
$\bm{\Lambda} \in \R^{r \times r}$ is a diagonal matrix of eigenvalues of the
low-rank Whittle correction. In this form, a slight manipulation and then
applying Theorem $3.1$ of \cite{ambikasaran2014} gives a very simple symmetric
factorization as
\begin{align} \label{eq:symfact}
  \notag
  \bm{D} + \bm{U} \bm{\Lambda} \bm{U}^H
  &=
  \bm{D}^{1/2}(\I + \tilde{\bm{U}} \bm{\Lambda} \tilde{\bm{U}}^H)\bm{D}^{1/2}
  \\
  &=
  \bm{D}^{1/2}
  (\I + \tilde{\bm{U}} \bm{X} \tilde{\bm{U}}^H)
  (\I + \tilde{\bm{U}} \bm{X} \tilde{\bm{U}}^H)^H
  \bm{D}^{1/2}
  =
  \bm{W} \bm{W}^H,
\end{align}
where $\tilde{\bm{U}} = \bm{D}^{-1/2} \bm{U}$, $\bm{X} = \bm{L}^{-H} (\bm{G}
- \I) \bm{L}^{-1}$, $\bm{U}^T \bm{U} = \bm{L} \bm{L}^T$, and $\I + \bm{L}^T
\bm{\Lambda} \bm{L} = \bm{G} \bm{G}^T$. While that is many equations, the actual
computation of the matrix $\bm{X}$ used in assembling $\bm{W}$ is is done in
just a few $\R^{r \times r}$ matrix-matrix operations (Algorithms $1$ and $2$ in
\cite{ambikasaran2014}).

This symmetric factor is useful for several purposes. For one, it gives a means
of obtaining exact simulations of any time series in $\bO(n \log n)$ time,
although of course standard periodic embedding would be a more simple and
efficient way of achieving the same thing. In the setting of derivative
information, it offers a more novel benefit in that it enables very accurate
stochastic estimation of Fisher information matrices.

The Fisher information matrix for parameters $\bth$ in a Gaussian model like the
one used in this work is given by
\begin{equation} \label{eq:fish}
  \bm{I}(\bth)_{j,k} = \frac{1}{2} \text{tr}\left(
    \tbS_{\bth}^{-1} 
    \set{
      \frac{\partial}{\partial \theta_j} \tbS_{\bth}
    }
    \tbS_{\bth}^{-1} 
    \set{
      \frac{\partial}{\partial \theta_k} \tbS_{\bth}
    }
  \right).
\end{equation}
This matrix is the asymptotic precision of the MLE under sufficient regularity
conditions, and is often substituted in place of the Hessian matrix of a
log-likelihood $\ell(\bth)$ due to the fact that it is easier to compute and is
always positive definite. As discussed above, it is fully possible to scalably
evaluate this sequence of matrix-matrix products to obtain an exact Fisher
information matrix, and in the following section we will provide a verification
that this computation still runs in quasilinear complexity despite the large
number of matrix-matrix operations. But \cite{geoga2020} introduced a fast
``symmetrized" stochastic estimator for the matrix $\bm{I}(\bth)$ that is very
accurate and can be computed in a fraction of the runtime cost of the exact
$\bm{I}(\bth)$.  Stochastic trace estimation is a rich field with a broad
literature (we refer readers to \cite{avron2011} for a useful introduction and
overview), but for this work the fundamental observation to make is that for any
random vector $\bm{u}$ with $\E \bu = \bm{0}$ and $\V \bu = \I$, we have that
$\E \bu^T \bm{A} \bu = \text{tr}(\bm{A})$ for any matrix $\bm{A}$.  The
\emph{sample average approximation} (SAA) trace estimator is then based on
drawing several i.i.d.  vectors $\bu$, denoted $\set{\bu_l}_{l=1}^L$, and
rather evaluating the trace directly, instead using 
\begin{equation*} 
  L^{-1} \sum_{l=1}^L \bu_l^T \bm{A} \bu_l \approx \text{tr}(\bm{A}).
\end{equation*}
The variance of this estimator depends on several properties of $\bm{A}$, but a
specific choice of $\bu$ having i.i.d. random signs is particularly popular
because $\V \bu^T \bm{A} \bu = 2(\norm[F]{\bm{A}}^2 - \sum_j A_{j,j}^2)$, so that
for diagonally concentrated matrices this estimator can be quite accurate. 

Stochastic trace estimation has been applied to the Gaussian process computation
problem in many works, notably first in \cite{anitescu2012}. In \cite{stein2013}, a
theorem was provided indicating that if $\tbS_{\bth} = \bm{W} \bm{W}^T$, then using the
simple property that $\text{tr}(\bm{A} \bm{B}) = \text{tr}(\bm{B} \bm{A})$
combined with this factorization, one can instead compute $\bm{I}_{j,k}$ with 
\begin{equation} \label{eq:symfish}
  \bm{I}_{j,k}(\bth) = \frac{1}{2} \text{tr}
  \left(
    \bm{W}^{-T} 
    \set{
      \frac{\partial}{\partial \theta_j} \tbS_{\bth}
    }
    \tbS_{\bth}^{-1} 
    \set{
      \frac{\partial}{\partial \theta_k} \tbS_{\bth}
    }
    \bm{W}^{-1} 
  \right),
\end{equation}
and the variance of the SAA estimator for (\ref{eq:symfish}) is bounded above by
the variance of the SAA estimator for (\ref{eq:fish}). This theoretical
observation was verified in \cite{geoga2020} and subsequent works, and
\cite{geoga2020} provided an even further ``symmetrized" trace to estimate for
off-diagonal elements. Introducing the notation $\tbS_j = \partial_{\theta_j}
\tbS_{\bth}$, one may compute a fully ``symmetrized" trace with 
\begin{equation} \label{eq:symsymfish}
  \widehat{\bm{I}}_{j,k}(\bth) = 
  (4 L)^{-1} \sum_{l=1}^L
  \bu_l^T \bm{W}^{-1} (\tbS_j + \tbS_k) \tbS^{-1} (\tbS_j + \tbS_k) \bm{W}^{-T} \bu_l
  - \frac{1}{2}\widehat{\bm{I}}_{j,j} - \frac{1}{2}\widehat{\bm{I}}_{k,k},
\end{equation}
where diagonal elements $\bm{I}_{j,j}$ are trivial to fully symmetrize and thus
can be computed in advance of the off-diagonal ones.

This estimator $\widehat{\bm{I}}$ enjoys both very high accuracy and great
computational benefits. The primary computational benefit is that this can be
computed in a single pass over the derivative matrices $\tbS_j$, since all one
needs to evaluate (\ref{eq:symsymfish}) is $\tbS^{-1}$ and $\set{\tbS_j
\bm{W}^{-T} \bu_l}_{l=1}^L$. Since one commonly uses $L \approx 70$ or some
similarly small number of SAA vectors, for example, it is very easy to pre-solve
the SAA vectors $\set{\bu_l}$ with $\bm{W}^{-T}$ and then in a single for-loop
assemble and apply $\tbS_j$ to each of them. Whether those pre-applied vectors
are saved to disk or kept in RAM, there is no more need for the derivative matrix
$\tbS_j$ after that point, and so one \new{does not need to} have \new{two
derivative matrices} $\tbS_j$ and $\tbS_k$ instantiated at the same time to
fully evaluate $\widehat{\bm{I}}$. For models with many parameters, the speedup
that this affords can be substantial. \new{Algorithm \ref{alg:sfish} summarizes
the computational procedure for computing (\ref{eq:symsymfish})}.

\begin{algorithm}
\caption{\new{Compute the stochastic expected Fisher matrix (\ref{eq:symsymfish})}}
\label{alg:sfish}
\begin{algorithmic}
  \STATE{\new{1. Using Algorithm \ref{alg:assemble}, create $\bm{D} + \bm{U}
  \bm{V^T}$}.}
  \STATE{\new{2. Using (\ref{eq:symfact}), compute $\bm{D} + \bm{U} \bm{V}^T =
  \bm{W} \bm{W}^T$}.}
  \STATE{\new{3. Draw SAA vectors $\set{\bv_l}_{l=1}^L$ with $[\bv_l]_k
  \overset{\text{i.i.d.}}{\sim} \text{Rad}(1/2)$}.}
  \WHILE{$j < p$}
    \STATE{\new{1. Using Algorithm \ref{alg:assemble} with
    $\frac{\partial}{\partial \theta_j} S_{\bth}(\omg)$, assemble 
    $\tbS_j = \bm{D}_j + \bm{U}_j \bm{V}_j^T$ from (\ref{eq:lrdj})}.}
    \STATE{\new{2. Compute $\set{\bv_{l,j} = \tbS_j \bm{W}^{-T}
    \bv_l}_{l=1}^L$}.}
  \ENDWHILE
  \STATE{\new{4. Using $\set{\bv_{l,j}}$, assemble (\ref{eq:symsymfish}) using only
  standard inner products}.}
\end{algorithmic}
\end{algorithm}

The Hessian of $\ell(\bth \sv \bv)$, while again absolutely computable in $\bO(n
\log n)$ complexity, is a similar situation to the expected Fisher information
matrix. It has terms given by
\begin{equation*} \label{eq:hess}
  [H \ell(\bth \sv \bv)]_{j,k}
  =
  \bm{I}_{j,k}
  + 
  \text{tr}\left(
    \tbS^{-1} \set{ \partial_{\theta_k} \partial_{\theta_j} \tbS }
  \right)
  -
  \bv^H \left( \partial_{\theta_k}
    \set{
      \tbS^{-1} \tbS_j \tbS^{-1} 
    }
  \right)
  \bv.
\end{equation*}
From this expression one sees that the same things are possible as with the
gradient, but the prefactor on all operations will be larger: one must compute
and re-compress matrix products involving second derivatives of $\bm{D} + \bm{U}
\bm{V}^H$ with respect to model parameters, and one must again do a full
matrix-matrix multiply for the additional trace (unless one opts to again use
SAA). All of these things are perfectly computable using the tools already
introduced in this work. But because of the high prefactor cost, we do not
explore them further here.

\section{Numerical demonstrations} \label{sec:demo}

In this section we will demonstrate the efficiency and accuracy of this method
for approximating $\bS$, the log-likelihood, and its derivatives. In particular,
we will study two spectral densities:
\begin{align*} 
  S_1(\omg) &= \theta_1 (1 - 2 \theta_2 \cos(2 \pi \omg) + \theta_2^2)^{-1} \\
  S_2(\omg) &= \theta_1 e^{-\theta_2|\omg|},
\end{align*}
which have corresponding autocovariance sequences
\begin{align*} 
  h_{1,k} &= \theta_1 (1 - \theta_2^2)^{-1} \theta_2^k  \\
  h_{2,k} &= \frac{2 \theta_1 e^{-\theta_2 /2}(-\theta_2 \cos(\pi k) + e^{\theta_2
  /2}\theta_2 + 2 \pi k \sin(\pi k))}{\theta_2^2 + (2 \pi k)^2}.
\end{align*}
The model $S_1$ was chosen because it is not only very smooth and with a limited
dynamic range, but its periodic extension is even continuous at the endpoints
(which the above section demonstrates are often the dominant source of
structure in the Whittle correction for otherwise smooth SDFs).
This spectral density is in some sense maximally well-behaved, and its
corresponding autocovariance sequence decays predictably quickly. From the
perspective of this approximation framework, this is an ideal setting.
$S_2(\omg)$, on the other hand, is more challenging because of its roughness at
the origin. This model was chosen to demonstrate that the directional derivative
correction introduced above can fully restore the accuracy of asymptotic
expansions even for spectral densities that aren't even always once
differentiable, and that the rank structure of $\bE$ is still very much exploitable.

As a first investigation, we verify the quasilinear runtime complexity of
assembling the matrix approximation and evaluating (\ref{eq:appx_nll}).
\begin{figure}[!ht]
  \centering
\begingroup
  \makeatletter
  \providecommand\color[2][]{%
    \GenericError{(gnuplot) \space\space\space\@spaces}{%
      Package color not loaded in conjunction with
      terminal option `colourtext'%
    }{See the gnuplot documentation for explanation.%
    }{Either use 'blacktext' in gnuplot or load the package
      color.sty in LaTeX.}%
    \renewcommand\color[2][]{}%
  }%
  \providecommand\includegraphics[2][]{%
    \GenericError{(gnuplot) \space\space\space\@spaces}{%
      Package graphicx or graphics not loaded%
    }{See the gnuplot documentation for explanation.%
    }{The gnuplot epslatex terminal needs graphicx.sty or graphics.sty.}%
    \renewcommand\includegraphics[2][]{}%
  }%
  \providecommand\rotatebox[2]{#2}%
  \@ifundefined{ifGPcolor}{%
    \newif\ifGPcolor
    \GPcolortrue
  }{}%
  \@ifundefined{ifGPblacktext}{%
    \newif\ifGPblacktext
    \GPblacktexttrue
  }{}%
  \let\gplgaddtomacro\g@addto@macro
  \gdef\gplbacktext{}%
  \gdef\gplfronttext{}%
  \makeatother
  \ifGPblacktext
    \def\colorrgb#1{}%
    \def\colorgray#1{}%
  \else
    \ifGPcolor
      \def\colorrgb#1{\color[rgb]{#1}}%
      \def\colorgray#1{\color[gray]{#1}}%
      \expandafter\def\csname LTw\endcsname{\color{white}}%
      \expandafter\def\csname LTb\endcsname{\color{black}}%
      \expandafter\def\csname LTa\endcsname{\color{black}}%
      \expandafter\def\csname LT0\endcsname{\color[rgb]{1,0,0}}%
      \expandafter\def\csname LT1\endcsname{\color[rgb]{0,1,0}}%
      \expandafter\def\csname LT2\endcsname{\color[rgb]{0,0,1}}%
      \expandafter\def\csname LT3\endcsname{\color[rgb]{1,0,1}}%
      \expandafter\def\csname LT4\endcsname{\color[rgb]{0,1,1}}%
      \expandafter\def\csname LT5\endcsname{\color[rgb]{1,1,0}}%
      \expandafter\def\csname LT6\endcsname{\color[rgb]{0,0,0}}%
      \expandafter\def\csname LT7\endcsname{\color[rgb]{1,0.3,0}}%
      \expandafter\def\csname LT8\endcsname{\color[rgb]{0.5,0.5,0.5}}%
    \else
      \def\colorrgb#1{\color{black}}%
      \def\colorgray#1{\color[gray]{#1}}%
      \expandafter\def\csname LTw\endcsname{\color{white}}%
      \expandafter\def\csname LTb\endcsname{\color{black}}%
      \expandafter\def\csname LTa\endcsname{\color{black}}%
      \expandafter\def\csname LT0\endcsname{\color{black}}%
      \expandafter\def\csname LT1\endcsname{\color{black}}%
      \expandafter\def\csname LT2\endcsname{\color{black}}%
      \expandafter\def\csname LT3\endcsname{\color{black}}%
      \expandafter\def\csname LT4\endcsname{\color{black}}%
      \expandafter\def\csname LT5\endcsname{\color{black}}%
      \expandafter\def\csname LT6\endcsname{\color{black}}%
      \expandafter\def\csname LT7\endcsname{\color{black}}%
      \expandafter\def\csname LT8\endcsname{\color{black}}%
    \fi
  \fi
    \setlength{\unitlength}{0.0500bp}%
    \ifx\gptboxheight\undefined%
      \newlength{\gptboxheight}%
      \newlength{\gptboxwidth}%
      \newsavebox{\gptboxtext}%
    \fi%
    \setlength{\fboxrule}{0.5pt}%
    \setlength{\fboxsep}{1pt}%
    \definecolor{tbcol}{rgb}{1,1,1}%
\begin{picture}(7360.00,2820.00)%
    \gplgaddtomacro\gplbacktext{%
      \csname LTb\endcsname
      \put(793,975){\makebox(0,0)[r]{\strut{}$10^{-2}$}}%
      \csname LTb\endcsname
      \put(793,1634){\makebox(0,0)[r]{\strut{}$10^{-1}$}}%
      \csname LTb\endcsname
      \put(793,2293){\makebox(0,0)[r]{\strut{}$10^{0}$}}%
      \csname LTb\endcsname
      \put(1031,527){\makebox(0,0){\strut{}$10^{3}$}}%
      \csname LTb\endcsname
      \put(4034,527){\makebox(0,0){\strut{}$10^{4}$}}%
      \csname LTb\endcsname
      \put(7037,527){\makebox(0,0){\strut{}$10^{5}$}}%
    }%
    \gplgaddtomacro\gplfronttext{%
      \csname LTb\endcsname
      \put(2020,2404){\makebox(0,0)[r]{\strut{}$\bO(n \log n)$}}%
      \csname LTb\endcsname
      \put(2020,2230){\makebox(0,0)[r]{\strut{}(debiased)}}%
      \csname LTb\endcsname
      \put(2020,2056){\makebox(0,0)[r]{\strut{}$r = 0$}}%
      \csname LTb\endcsname
      \put(2020,1882){\makebox(0,0)[r]{\strut{}$r = 2$}}%
      \csname LTb\endcsname
      \put(3287,2404){\makebox(0,0)[r]{\strut{}$r = 32$}}%
      \csname LTb\endcsname
      \put(3287,2230){\makebox(0,0)[r]{\strut{}$r = 64$}}%
      \csname LTb\endcsname
      \put(3287,2056){\makebox(0,0)[r]{\strut{}$r = 128$}}%
      \csname LTb\endcsname
      \put(195,1663){\rotatebox{-270.00}{\makebox(0,0){\strut{}Time (s)}}}%
      \csname LTb\endcsname
      \put(3966,167){\makebox(0,0){\strut{}n}}%
    }%
    \gplbacktext
    \put(0,0){\includegraphics[width={368.00bp},height={141.00bp}]{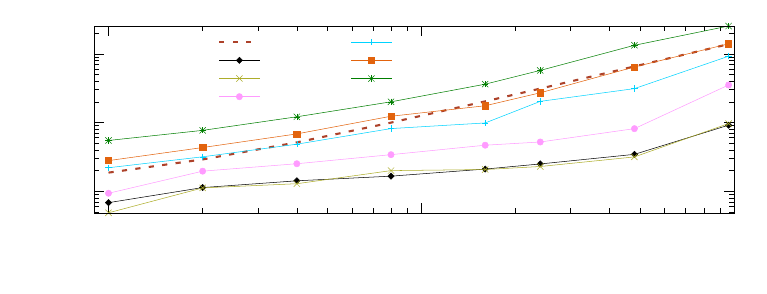}}%
    \gplfronttext
  \end{picture}%
\endgroup

  \vspace{-0.1in}

  \caption{The runtime cost of assembling the approximation $\Fr \bS \Fr^H
  \approx \bm{D} + \bU \bV^H$ and evaluating (\ref{eq:appx_nll}) for various
  different ranks $r$. We also provide the runtime cost of
  assembling just the diagonal matrix of $S(\omg)$ at Fourier frequencies
  ($r=0$) and the debiased Whittle estimator.}
  \label{fig:runtime}
\end{figure}
Figure \ref{fig:runtime} shows the runtime cost of those computations for a
variety of data sizes ranging from $n=1\,000$ to $n=96\,000$ along with a
theoretical $\bO(n \log n)$ line. As one would expect for a fixed-rank
approximation assembled using $\bO(1)$ FFTs, the empirical agreement with the
theoretical complexity is good. 

As a second investigation, we look at the relative error \new{in the
approximated negative log-likelihood $\tilde{\ell}(\bth \sv \tby)$}.  This
study, summarized in Figure \ref{fig:nll_relerr}, clearly demonstrates the way
that this approximation depends on the specific spectral density being modeled.
\new{All log-likelihoods are evaluated at a vector of i.i.d. standard normal
variables, and the parameters are chosen to provide challenging members of each
parametric family, with $(\theta_1, \theta_2) = (1, 0.9)$ for $S_1$ and
$(\theta_1, \theta_2) = (10, 10)$ for $S_2$}.  For $S_1$, we see that even a
rank of $r=2$ provides an effectively exact log-likelihood with $14+$ digits of
precision. The second two sub-plots in Figure \ref{fig:nll_relerr} show two
different applications with $S_2$. The center figure shows the relative accuracy
of the negative log-likelihood when one does \emph{not} correct the asymptotic
expansion for the rough point at the origin when computing the autocovariance
sequence $\set{h_k}_{k=0}^{n-1}$, instead simply pretending that the function
has many derivatives everywhere.  In this work the asymptotic expansions were
used for $k > 2\,000$, and it is clear in the figure that as soon as those
expansions start being used, one goes from $14$ digits in the rank $r=128$
case to $8$ digits, and the quality of the approximation does not meaningfully
improve even as the rank varies by an order of magnitude.
\begin{figure}[!ht]
  \centering
\begingroup
  \makeatletter
  \providecommand\color[2][]{%
    \GenericError{(gnuplot) \space\space\space\@spaces}{%
      Package color not loaded in conjunction with
      terminal option `colourtext'%
    }{See the gnuplot documentation for explanation.%
    }{Either use 'blacktext' in gnuplot or load the package
      color.sty in LaTeX.}%
    \renewcommand\color[2][]{}%
  }%
  \providecommand\includegraphics[2][]{%
    \GenericError{(gnuplot) \space\space\space\@spaces}{%
      Package graphicx or graphics not loaded%
    }{See the gnuplot documentation for explanation.%
    }{The gnuplot epslatex terminal needs graphicx.sty or graphics.sty.}%
    \renewcommand\includegraphics[2][]{}%
  }%
  \providecommand\rotatebox[2]{#2}%
  \@ifundefined{ifGPcolor}{%
    \newif\ifGPcolor
    \GPcolortrue
  }{}%
  \@ifundefined{ifGPblacktext}{%
    \newif\ifGPblacktext
    \GPblacktexttrue
  }{}%
  \let\gplgaddtomacro\g@addto@macro
  \gdef\gplbacktext{}%
  \gdef\gplfronttext{}%
  \makeatother
  \ifGPblacktext
    \def\colorrgb#1{}%
    \def\colorgray#1{}%
  \else
    \ifGPcolor
      \def\colorrgb#1{\color[rgb]{#1}}%
      \def\colorgray#1{\color[gray]{#1}}%
      \expandafter\def\csname LTw\endcsname{\color{white}}%
      \expandafter\def\csname LTb\endcsname{\color{black}}%
      \expandafter\def\csname LTa\endcsname{\color{black}}%
      \expandafter\def\csname LT0\endcsname{\color[rgb]{1,0,0}}%
      \expandafter\def\csname LT1\endcsname{\color[rgb]{0,1,0}}%
      \expandafter\def\csname LT2\endcsname{\color[rgb]{0,0,1}}%
      \expandafter\def\csname LT3\endcsname{\color[rgb]{1,0,1}}%
      \expandafter\def\csname LT4\endcsname{\color[rgb]{0,1,1}}%
      \expandafter\def\csname LT5\endcsname{\color[rgb]{1,1,0}}%
      \expandafter\def\csname LT6\endcsname{\color[rgb]{0,0,0}}%
      \expandafter\def\csname LT7\endcsname{\color[rgb]{1,0.3,0}}%
      \expandafter\def\csname LT8\endcsname{\color[rgb]{0.5,0.5,0.5}}%
    \else
      \def\colorrgb#1{\color{black}}%
      \def\colorgray#1{\color[gray]{#1}}%
      \expandafter\def\csname LTw\endcsname{\color{white}}%
      \expandafter\def\csname LTb\endcsname{\color{black}}%
      \expandafter\def\csname LTa\endcsname{\color{black}}%
      \expandafter\def\csname LT0\endcsname{\color{black}}%
      \expandafter\def\csname LT1\endcsname{\color{black}}%
      \expandafter\def\csname LT2\endcsname{\color{black}}%
      \expandafter\def\csname LT3\endcsname{\color{black}}%
      \expandafter\def\csname LT4\endcsname{\color{black}}%
      \expandafter\def\csname LT5\endcsname{\color{black}}%
      \expandafter\def\csname LT6\endcsname{\color{black}}%
      \expandafter\def\csname LT7\endcsname{\color{black}}%
      \expandafter\def\csname LT8\endcsname{\color{black}}%
    \fi
  \fi
    \setlength{\unitlength}{0.0500bp}%
    \ifx\gptboxheight\undefined%
      \newlength{\gptboxheight}%
      \newlength{\gptboxwidth}%
      \newsavebox{\gptboxtext}%
    \fi%
    \setlength{\fboxrule}{0.5pt}%
    \setlength{\fboxsep}{1pt}%
    \definecolor{tbcol}{rgb}{1,1,1}%
\begin{picture}(7360.00,3400.00)%
    \gplgaddtomacro\gplbacktext{%
      \csname LTb\endcsname
      \put(633,1030){\makebox(0,0)[r]{\strut{}$10^{-14}$}}%
      \csname LTb\endcsname
      \put(633,1300){\makebox(0,0)[r]{\strut{}$10^{-12}$}}%
      \csname LTb\endcsname
      \put(633,1570){\makebox(0,0)[r]{\strut{}$10^{-10}$}}%
      \csname LTb\endcsname
      \put(633,1840){\makebox(0,0)[r]{\strut{}$10^{-8}$}}%
      \csname LTb\endcsname
      \put(633,2110){\makebox(0,0)[r]{\strut{}$10^{-6}$}}%
      \csname LTb\endcsname
      \put(633,2380){\makebox(0,0)[r]{\strut{}$10^{-4}$}}%
      \csname LTb\endcsname
      \put(633,2650){\makebox(0,0)[r]{\strut{}$10^{-2}$}}%
      \csname LTb\endcsname
      \put(772,520){\makebox(0,0){\strut{}$10^{3}$}}%
      \csname LTb\endcsname
      \put(1609,520){\makebox(0,0){\strut{}$10^{4}$}}%
      \csname LTb\endcsname
      \put(2446,520){\makebox(0,0){\strut{}$10^{5}$}}%
    }%
    \gplgaddtomacro\gplfronttext{%
      \csname LTb\endcsname
      \put(1825,2028){\makebox(0,0)[r]{\strut{}(debiased)}}%
      \csname LTb\endcsname
      \put(1825,1865){\makebox(0,0)[r]{\strut{}$r = 0$}}%
      \csname LTb\endcsname
      \put(1825,1702){\makebox(0,0)[r]{\strut{}$r = 2$}}%
      \csname LTb\endcsname
      \put(1825,1538){\makebox(0,0)[r]{\strut{}$r = 32$}}%
      \csname LTb\endcsname
      \put(1825,1375){\makebox(0,0)[r]{\strut{}$r = 64$}}%
      \csname LTb\endcsname
      \put(1825,1212){\makebox(0,0)[r]{\strut{}$r = 128$}}%
      \csname LTb\endcsname
      \put(1590,161){\makebox(0,0){\strut{}n}}%
      \csname LTb\endcsname
      \put(1590,3063){\makebox(0,0){\strut{}$S_1$}}%
    }%
    \gplgaddtomacro\gplbacktext{%
      \csname LTb\endcsname
      \put(2712,1030){\makebox(0,0)[r]{\strut{}}}%
      \csname LTb\endcsname
      \put(2712,1300){\makebox(0,0)[r]{\strut{}}}%
      \csname LTb\endcsname
      \put(2712,1570){\makebox(0,0)[r]{\strut{}}}%
      \csname LTb\endcsname
      \put(2712,1840){\makebox(0,0)[r]{\strut{}}}%
      \csname LTb\endcsname
      \put(2712,2110){\makebox(0,0)[r]{\strut{}}}%
      \csname LTb\endcsname
      \put(2712,2380){\makebox(0,0)[r]{\strut{}}}%
      \csname LTb\endcsname
      \put(2712,2650){\makebox(0,0)[r]{\strut{}}}%
      \csname LTb\endcsname
      \put(2851,520){\makebox(0,0){\strut{}$10^{3}$}}%
      \csname LTb\endcsname
      \put(3689,520){\makebox(0,0){\strut{}$10^{4}$}}%
      \csname LTb\endcsname
      \put(4526,520){\makebox(0,0){\strut{}$10^{5}$}}%
    }%
    \gplgaddtomacro\gplfronttext{%
      \csname LTb\endcsname
      \put(3669,161){\makebox(0,0){\strut{}n}}%
      \csname LTb\endcsname
      \put(3669,3063){\makebox(0,0){\strut{}$S_2$,  no splitting at 0}}%
    }%
    \gplgaddtomacro\gplbacktext{%
      \csname LTb\endcsname
      \put(4792,1030){\makebox(0,0)[r]{\strut{}}}%
      \csname LTb\endcsname
      \put(4792,1300){\makebox(0,0)[r]{\strut{}}}%
      \csname LTb\endcsname
      \put(4792,1570){\makebox(0,0)[r]{\strut{}}}%
      \csname LTb\endcsname
      \put(4792,1840){\makebox(0,0)[r]{\strut{}}}%
      \csname LTb\endcsname
      \put(4792,2110){\makebox(0,0)[r]{\strut{}}}%
      \csname LTb\endcsname
      \put(4792,2380){\makebox(0,0)[r]{\strut{}}}%
      \csname LTb\endcsname
      \put(4792,2650){\makebox(0,0)[r]{\strut{}}}%
      \csname LTb\endcsname
      \put(4931,520){\makebox(0,0){\strut{}$10^{3}$}}%
      \csname LTb\endcsname
      \put(5768,520){\makebox(0,0){\strut{}$10^{4}$}}%
      \csname LTb\endcsname
      \put(6605,520){\makebox(0,0){\strut{}$10^{5}$}}%
      \csname LTb\endcsname
      \put(6706,1030){\makebox(0,0)[l]{\strut{}$10^{-14}$}}%
      \csname LTb\endcsname
      \put(6706,1300){\makebox(0,0)[l]{\strut{}$10^{-12}$}}%
      \csname LTb\endcsname
      \put(6706,1570){\makebox(0,0)[l]{\strut{}$10^{-10}$}}%
      \csname LTb\endcsname
      \put(6706,1840){\makebox(0,0)[l]{\strut{}$10^{-8}$}}%
      \csname LTb\endcsname
      \put(6706,2110){\makebox(0,0)[l]{\strut{}$10^{-6}$}}%
      \csname LTb\endcsname
      \put(6706,2380){\makebox(0,0)[l]{\strut{}$10^{-4}$}}%
      \csname LTb\endcsname
      \put(6706,2650){\makebox(0,0)[l]{\strut{}$10^{-2}$}}%
    }%
    \gplgaddtomacro\gplfronttext{%
      \csname LTb\endcsname
      \put(5749,161){\makebox(0,0){\strut{}n}}%
      \csname LTb\endcsname
      \put(5749,3063){\makebox(0,0){\strut{}$S_2$,  split at 0}}%
    }%
    \gplbacktext
    \put(0,0){\includegraphics[width={368.00bp},height={170.00bp}]{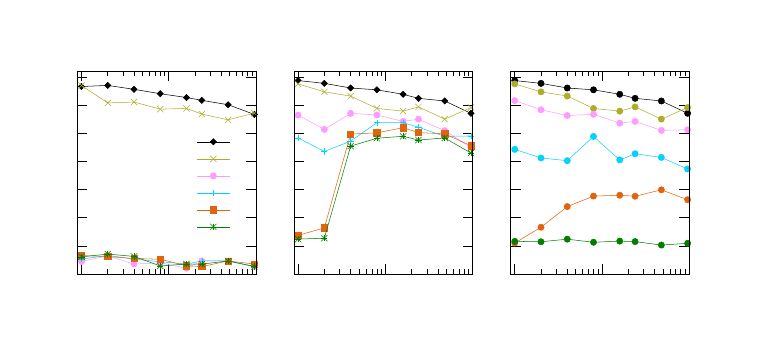}}%
    \gplfronttext
  \end{picture}%
\endgroup

  \vspace{-0.1in}

  \caption{The relative errors $|\ell(\bth \sv \by) - \tilde{\ell}(\bth \sv
  \tby)|/|\ell(\bth \sv \by)|$ for
  various models, ranks, and data sizes. The middle column shows the error in
  naively applying the asymptotic expansion method of Proposition
  \ref{thm:asexp} to obtain $\set{h_k}_{k=0}^{n-1}$ despite $S_2(\omg)$ having
  zero derivatives at the origin, and the right column applies the split
  expansion method from Corollary \ref{cor:asexp}.
  }
  \label{fig:nll_relerr}
\end{figure}

The final panel in Figure \ref{fig:nll_relerr} shows the relative error of the
likelihood approximation with the corrected asymptotic expansion that applies
Corollary \ref{cor:asexp} at the splitting point $\omg_1^r = 0$.  Here we now
see that there is no visible loss in accuracy once the expansions start being
used for tails of the autocovariance sequence, indicating that the
domain-splitting correction restores the accuracy of the expansions. As the
theory would imply and  this figure empirically demonstrates, the rank of the
Whittle correction matrix $\bE$ depends on smoothness properties of the spectral
density.  In this particular case, a rank of $r \approx 100$ is necessary to get
almost every significant digit of the log-likelihood. This low rank makes sense
in light of Section \ref{sec:lowrank_theory}'s asymptotic expansion argument, as
$S_2(\omg)$ is non-smooth at the origin and at the periodic endpoints.
\begin{figure}[!ht]
  \centering
\begingroup
  \makeatletter
  \providecommand\color[2][]{%
    \GenericError{(gnuplot) \space\space\space\@spaces}{%
      Package color not loaded in conjunction with
      terminal option `colourtext'%
    }{See the gnuplot documentation for explanation.%
    }{Either use 'blacktext' in gnuplot or load the package
      color.sty in LaTeX.}%
    \renewcommand\color[2][]{}%
  }%
  \providecommand\includegraphics[2][]{%
    \GenericError{(gnuplot) \space\space\space\@spaces}{%
      Package graphicx or graphics not loaded%
    }{See the gnuplot documentation for explanation.%
    }{The gnuplot epslatex terminal needs graphicx.sty or graphics.sty.}%
    \renewcommand\includegraphics[2][]{}%
  }%
  \providecommand\rotatebox[2]{#2}%
  \@ifundefined{ifGPcolor}{%
    \newif\ifGPcolor
    \GPcolortrue
  }{}%
  \@ifundefined{ifGPblacktext}{%
    \newif\ifGPblacktext
    \GPblacktexttrue
  }{}%
  \let\gplgaddtomacro\g@addto@macro
  \gdef\gplbacktext{}%
  \gdef\gplfronttext{}%
  \makeatother
  \ifGPblacktext
    \def\colorrgb#1{}%
    \def\colorgray#1{}%
  \else
    \ifGPcolor
      \def\colorrgb#1{\color[rgb]{#1}}%
      \def\colorgray#1{\color[gray]{#1}}%
      \expandafter\def\csname LTw\endcsname{\color{white}}%
      \expandafter\def\csname LTb\endcsname{\color{black}}%
      \expandafter\def\csname LTa\endcsname{\color{black}}%
      \expandafter\def\csname LT0\endcsname{\color[rgb]{1,0,0}}%
      \expandafter\def\csname LT1\endcsname{\color[rgb]{0,1,0}}%
      \expandafter\def\csname LT2\endcsname{\color[rgb]{0,0,1}}%
      \expandafter\def\csname LT3\endcsname{\color[rgb]{1,0,1}}%
      \expandafter\def\csname LT4\endcsname{\color[rgb]{0,1,1}}%
      \expandafter\def\csname LT5\endcsname{\color[rgb]{1,1,0}}%
      \expandafter\def\csname LT6\endcsname{\color[rgb]{0,0,0}}%
      \expandafter\def\csname LT7\endcsname{\color[rgb]{1,0.3,0}}%
      \expandafter\def\csname LT8\endcsname{\color[rgb]{0.5,0.5,0.5}}%
    \else
      \def\colorrgb#1{\color{black}}%
      \def\colorgray#1{\color[gray]{#1}}%
      \expandafter\def\csname LTw\endcsname{\color{white}}%
      \expandafter\def\csname LTb\endcsname{\color{black}}%
      \expandafter\def\csname LTa\endcsname{\color{black}}%
      \expandafter\def\csname LT0\endcsname{\color{black}}%
      \expandafter\def\csname LT1\endcsname{\color{black}}%
      \expandafter\def\csname LT2\endcsname{\color{black}}%
      \expandafter\def\csname LT3\endcsname{\color{black}}%
      \expandafter\def\csname LT4\endcsname{\color{black}}%
      \expandafter\def\csname LT5\endcsname{\color{black}}%
      \expandafter\def\csname LT6\endcsname{\color{black}}%
      \expandafter\def\csname LT7\endcsname{\color{black}}%
      \expandafter\def\csname LT8\endcsname{\color{black}}%
    \fi
  \fi
    \setlength{\unitlength}{0.0500bp}%
    \ifx\gptboxheight\undefined%
      \newlength{\gptboxheight}%
      \newlength{\gptboxwidth}%
      \newsavebox{\gptboxtext}%
    \fi%
    \setlength{\fboxrule}{0.5pt}%
    \setlength{\fboxsep}{1pt}%
    \definecolor{tbcol}{rgb}{1,1,1}%
\begin{picture}(7360.00,2820.00)%
    \gplgaddtomacro\gplbacktext{%
      \csname LTb\endcsname
      \put(523,941){\makebox(0,0)[r]{\strut{}$0.1$}}%
      \csname LTb\endcsname
      \put(523,1459){\makebox(0,0)[r]{\strut{}$1$}}%
      \csname LTb\endcsname
      \put(523,1977){\makebox(0,0)[r]{\strut{}$10$}}%
      \csname LTb\endcsname
      \put(523,2288){\makebox(0,0)[r]{\strut{}$40$}}%
      \csname LTb\endcsname
      \put(673,390){\makebox(0,0){\strut{}$10^{3}$}}%
      \csname LTb\endcsname
      \put(1767,390){\makebox(0,0){\strut{}$10^{4}$}}%
    }%
    \gplgaddtomacro\gplfronttext{%
      \csname LTb\endcsname
      \put(25,1504){\rotatebox{-270.00}{\makebox(0,0){\strut{}Time (s)}}}%
      \csname LTb\endcsname
      \put(1578,30){\makebox(0,0){\strut{}n}}%
      \csname LTb\endcsname
      \put(1578,2739){\makebox(0,0){\strut{}gradient}}%
    }%
    \gplgaddtomacro\gplbacktext{%
      \csname LTb\endcsname
      \put(2615,941){\makebox(0,0)[r]{\strut{}}}%
      \csname LTb\endcsname
      \put(2615,1459){\makebox(0,0)[r]{\strut{}}}%
      \csname LTb\endcsname
      \put(2615,1977){\makebox(0,0)[r]{\strut{}}}%
      \csname LTb\endcsname
      \put(2615,2288){\makebox(0,0)[r]{\strut{}}}%
      \csname LTb\endcsname
      \put(2765,390){\makebox(0,0){\strut{}$10^{3}$}}%
      \csname LTb\endcsname
      \put(3859,390){\makebox(0,0){\strut{}$10^{4}$}}%
      \csname LTb\endcsname
      \put(4724,941){\makebox(0,0)[l]{\strut{}}}%
      \csname LTb\endcsname
      \put(4724,1459){\makebox(0,0)[l]{\strut{}}}%
      \csname LTb\endcsname
      \put(4724,1977){\makebox(0,0)[l]{\strut{}}}%
      \csname LTb\endcsname
      \put(4724,2288){\makebox(0,0)[l]{\strut{}}}%
    }%
    \gplgaddtomacro\gplfronttext{%
      \csname LTb\endcsname
      \put(3669,30){\makebox(0,0){\strut{}n}}%
      \csname LTb\endcsname
      \put(3669,2739){\makebox(0,0){\strut{}exact exp. Fisher}}%
    }%
    \gplgaddtomacro\gplbacktext{%
      \csname LTb\endcsname
      \put(4707,941){\makebox(0,0)[r]{\strut{}}}%
      \csname LTb\endcsname
      \put(4707,1459){\makebox(0,0)[r]{\strut{}}}%
      \csname LTb\endcsname
      \put(4707,1977){\makebox(0,0)[r]{\strut{}}}%
      \csname LTb\endcsname
      \put(4707,2288){\makebox(0,0)[r]{\strut{}}}%
      \csname LTb\endcsname
      \put(4857,390){\makebox(0,0){\strut{}$10^{3}$}}%
      \csname LTb\endcsname
      \put(5951,390){\makebox(0,0){\strut{}$10^{4}$}}%
      \csname LTb\endcsname
      \put(6816,941){\makebox(0,0)[l]{\strut{}$0.1$}}%
      \csname LTb\endcsname
      \put(6816,1459){\makebox(0,0)[l]{\strut{}$1$}}%
      \csname LTb\endcsname
      \put(6816,1977){\makebox(0,0)[l]{\strut{}$10$}}%
      \csname LTb\endcsname
      \put(6816,2288){\makebox(0,0)[l]{\strut{}$40$}}%
    }%
    \gplgaddtomacro\gplfronttext{%
      \csname LTb\endcsname
      \put(5694,2238){\makebox(0,0)[r]{\strut{}$\bO(n \log n)$}}%
      \csname LTb\endcsname
      \put(5694,2086){\makebox(0,0)[r]{\strut{}$r = 2$}}%
      \csname LTb\endcsname
      \put(5694,1933){\makebox(0,0)[r]{\strut{}$r = 4$}}%
      \csname LTb\endcsname
      \put(5694,1781){\makebox(0,0)[r]{\strut{}$r = 32$}}%
      \csname LTb\endcsname
      \put(5694,1628){\makebox(0,0)[r]{\strut{}$r = 64$}}%
      \csname LTb\endcsname
      \put(5694,1476){\makebox(0,0)[r]{\strut{}$r = 128$}}%
      \csname LTb\endcsname
      \put(5761,30){\makebox(0,0){\strut{}n}}%
      \csname LTb\endcsname
      \put(5761,2739){\makebox(0,0){\strut{}stoch. exp. Fisher}}%
    }%
    \gplbacktext
    \put(0,0){\includegraphics[width={368.00bp},height={141.00bp}]{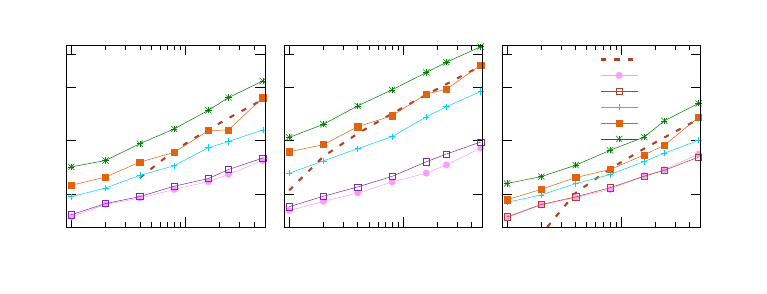}}%
    \gplfronttext
  \end{picture}%
\endgroup

  \vspace{-0.05in}

  \caption{Runtime costs for evaluating the approximated gradient
  (\ref{eq:grad}), the exact Fisher information matrix (\ref{eq:fish}), and the
  symmetrized stochastic expected Fisher information matrix for various data
  sizes and ranks using $S_2(\omg)$ with asymptotic expansions split at $0$.}
  \label{fig:gradfish_runtime}
\end{figure}

Next, we validate the runtime cost and accuracy of the gradient, the exact
expected Fisher matrix, and the symmetrized stochastic expected Fisher matrix of
the negative log-likelihood, although due to the exact values requiring
matrix-matrix products we now only test to matrices of size $48\,000$.  Figure
\ref{fig:gradfish_runtime} provides a summary of the runtime cost of computing
the three quantities. As can be seen, the agreement with the theoretically
expected $\bO(n \log n)$ is clear. And even for $n \approx 50\,000$ data points,
we see that the gradient of the likelihood---which requires a matrix-matrix
product for each term---can be computed to full precision in under ten seconds.
By virtue of requiring only matrix-vector products, which are exceptionally fast
with this simple low-rank structure, the symmetrized stochastic expected Fisher
information matrix is actually slightly faster to compute than the gradient.
Each term of the exact expected Fisher information matrix requires a product of
four matrices, and so the prefactor of the $\bO(n \log n)$ complexity is
significantly higher, with the rank $r=128$ value for $n=48\,000$ taking about
$40$ seconds compared to the $\approx 5 \text{s}$ that its stochastic
counterpart required. This is of course more expensive, but it is also
sufficiently manageable that one could certainly compute it once to high
accuracy to obtain the true asymptotic precision of the MLE once the estimate
$\hat{\bth}^{\text{MLE}}$ has been computed.

To study the accuracy of these approximations, Figure \ref{fig:gradfish_err} shows
two quantities. Letting $\bm{g}(\bth)$ denote the true gradient $\nabla
\ell(\bth \sv \by)$ and $\tilde{\bm{g}}(\bth)$ denote the approximated gradient
$\nabla \tilde{\ell}(\bth \sv \by)$, the top row shows a standard max-norm type
metric of 
\begin{equation*} 
  \max_{j=1, ..., p} \frac{\abs{g_j(\bth) -
  \tilde{g}_j(\bth)}}{\abs{g_j(\bth)}}
\end{equation*}
for models $S_1(\omg)$ and $S_2(\omg)$.  Letting $\bm{I}$ denote the true
expected Fisher information matrix and $\tilde{\bm{I}}$ denote the approximated
one (where the additional approximation layer of stochastic trace estimation is
marked in the legend), the second row of Figure \ref{fig:gradfish_err} shows the
relative operator norm error
$
  \norm{\bm{I}}^{-1}\norm{\bm{I} - \tilde{\bm{I}}},
$
which is arguably a more relevant metric for a Hessian approximation than a
pointwise max-norm. 
\begin{figure}[!ht]
  \centering
\begingroup
  \makeatletter
  \providecommand\color[2][]{%
    \GenericError{(gnuplot) \space\space\space\@spaces}{%
      Package color not loaded in conjunction with
      terminal option `colourtext'%
    }{See the gnuplot documentation for explanation.%
    }{Either use 'blacktext' in gnuplot or load the package
      color.sty in LaTeX.}%
    \renewcommand\color[2][]{}%
  }%
  \providecommand\includegraphics[2][]{%
    \GenericError{(gnuplot) \space\space\space\@spaces}{%
      Package graphicx or graphics not loaded%
    }{See the gnuplot documentation for explanation.%
    }{The gnuplot epslatex terminal needs graphicx.sty or graphics.sty.}%
    \renewcommand\includegraphics[2][]{}%
  }%
  \providecommand\rotatebox[2]{#2}%
  \@ifundefined{ifGPcolor}{%
    \newif\ifGPcolor
    \GPcolortrue
  }{}%
  \@ifundefined{ifGPblacktext}{%
    \newif\ifGPblacktext
    \GPblacktexttrue
  }{}%
  \let\gplgaddtomacro\g@addto@macro
  \gdef\gplbacktext{}%
  \gdef\gplfronttext{}%
  \makeatother
  \ifGPblacktext
    \def\colorrgb#1{}%
    \def\colorgray#1{}%
  \else
    \ifGPcolor
      \def\colorrgb#1{\color[rgb]{#1}}%
      \def\colorgray#1{\color[gray]{#1}}%
      \expandafter\def\csname LTw\endcsname{\color{white}}%
      \expandafter\def\csname LTb\endcsname{\color{black}}%
      \expandafter\def\csname LTa\endcsname{\color{black}}%
      \expandafter\def\csname LT0\endcsname{\color[rgb]{1,0,0}}%
      \expandafter\def\csname LT1\endcsname{\color[rgb]{0,1,0}}%
      \expandafter\def\csname LT2\endcsname{\color[rgb]{0,0,1}}%
      \expandafter\def\csname LT3\endcsname{\color[rgb]{1,0,1}}%
      \expandafter\def\csname LT4\endcsname{\color[rgb]{0,1,1}}%
      \expandafter\def\csname LT5\endcsname{\color[rgb]{1,1,0}}%
      \expandafter\def\csname LT6\endcsname{\color[rgb]{0,0,0}}%
      \expandafter\def\csname LT7\endcsname{\color[rgb]{1,0.3,0}}%
      \expandafter\def\csname LT8\endcsname{\color[rgb]{0.5,0.5,0.5}}%
    \else
      \def\colorrgb#1{\color{black}}%
      \def\colorgray#1{\color[gray]{#1}}%
      \expandafter\def\csname LTw\endcsname{\color{white}}%
      \expandafter\def\csname LTb\endcsname{\color{black}}%
      \expandafter\def\csname LTa\endcsname{\color{black}}%
      \expandafter\def\csname LT0\endcsname{\color{black}}%
      \expandafter\def\csname LT1\endcsname{\color{black}}%
      \expandafter\def\csname LT2\endcsname{\color{black}}%
      \expandafter\def\csname LT3\endcsname{\color{black}}%
      \expandafter\def\csname LT4\endcsname{\color{black}}%
      \expandafter\def\csname LT5\endcsname{\color{black}}%
      \expandafter\def\csname LT6\endcsname{\color{black}}%
      \expandafter\def\csname LT7\endcsname{\color{black}}%
      \expandafter\def\csname LT8\endcsname{\color{black}}%
    \fi
  \fi
    \setlength{\unitlength}{0.0500bp}%
    \ifx\gptboxheight\undefined%
      \newlength{\gptboxheight}%
      \newlength{\gptboxwidth}%
      \newsavebox{\gptboxtext}%
    \fi%
    \setlength{\fboxrule}{0.5pt}%
    \setlength{\fboxsep}{1pt}%
    \definecolor{tbcol}{rgb}{1,1,1}%
\begin{picture}(7360.00,5660.00)%
    \gplgaddtomacro\gplbacktext{%
      \csname LTb\endcsname
      \put(633,3431){\makebox(0,0)[r]{\strut{}$10^{-15}$}}%
      \csname LTb\endcsname
      \put(633,3783){\makebox(0,0)[r]{\strut{}$10^{-12}$}}%
      \csname LTb\endcsname
      \put(633,4136){\makebox(0,0)[r]{\strut{}$10^{-9}$}}%
      \csname LTb\endcsname
      \put(633,4488){\makebox(0,0)[r]{\strut{}$10^{-6}$}}%
      \csname LTb\endcsname
      \put(633,4840){\makebox(0,0)[r]{\strut{}$10^{-3}$}}%
      \csname LTb\endcsname
      \put(807,3073){\makebox(0,0){\strut{}}}%
      \csname LTb\endcsname
      \put(2416,3073){\makebox(0,0){\strut{}}}%
    }%
    \gplgaddtomacro\gplfronttext{%
      \csname LTb\endcsname
      \put(2137,5435){\makebox(0,0){\strut{}$S_1$ gradient}}%
    }%
    \gplgaddtomacro\gplbacktext{%
      \csname LTb\endcsname
      \put(3697,3431){\makebox(0,0)[r]{\strut{}}}%
      \csname LTb\endcsname
      \put(3697,3783){\makebox(0,0)[r]{\strut{}}}%
      \csname LTb\endcsname
      \put(3697,4136){\makebox(0,0)[r]{\strut{}}}%
      \csname LTb\endcsname
      \put(3697,4488){\makebox(0,0)[r]{\strut{}}}%
      \csname LTb\endcsname
      \put(3697,4840){\makebox(0,0)[r]{\strut{}}}%
      \csname LTb\endcsname
      \put(3872,3073){\makebox(0,0){\strut{}}}%
      \csname LTb\endcsname
      \put(5481,3073){\makebox(0,0){\strut{}}}%
      \csname LTb\endcsname
      \put(6706,3431){\makebox(0,0)[l]{\strut{}$10^{-15}$}}%
      \csname LTb\endcsname
      \put(6706,3783){\makebox(0,0)[l]{\strut{}$10^{-12}$}}%
      \csname LTb\endcsname
      \put(6706,4136){\makebox(0,0)[l]{\strut{}$10^{-9}$}}%
      \csname LTb\endcsname
      \put(6706,4488){\makebox(0,0)[l]{\strut{}$10^{-6}$}}%
      \csname LTb\endcsname
      \put(6706,4840){\makebox(0,0)[l]{\strut{}$10^{-3}$}}%
    }%
    \gplgaddtomacro\gplfronttext{%
      \csname LTb\endcsname
      \put(5202,5435){\makebox(0,0){\strut{}$S_2$ (split at 0) gradient}}%
    }%
    \gplgaddtomacro\gplbacktext{%
      \csname LTb\endcsname
      \put(633,822){\makebox(0,0)[r]{\strut{}$10^{-15}$}}%
      \csname LTb\endcsname
      \put(633,1175){\makebox(0,0)[r]{\strut{}$10^{-12}$}}%
      \csname LTb\endcsname
      \put(633,1527){\makebox(0,0)[r]{\strut{}$10^{-9}$}}%
      \csname LTb\endcsname
      \put(633,1879){\makebox(0,0)[r]{\strut{}$10^{-6}$}}%
      \csname LTb\endcsname
      \put(633,2232){\makebox(0,0)[r]{\strut{}$10^{-3}$}}%
      \csname LTb\endcsname
      \put(807,465){\makebox(0,0){\strut{}$10^{3}$}}%
      \csname LTb\endcsname
      \put(2416,465){\makebox(0,0){\strut{}$10^{4}$}}%
    }%
    \gplgaddtomacro\gplfronttext{%
      \csname LTb\endcsname
      \put(1351,1770){\makebox(0,0)[r]{\strut{}$r = 2$}}%
      \csname LTb\endcsname
      \put(1351,1601){\makebox(0,0)[r]{\strut{}$r = 4$}}%
      \csname LTb\endcsname
      \put(1351,1431){\makebox(0,0)[r]{\strut{}$r = 32$}}%
      \csname LTb\endcsname
      \put(1351,1261){\makebox(0,0)[r]{\strut{}$r = 64$}}%
      \csname LTb\endcsname
      \put(3094,1770){\makebox(0,0)[r]{\strut{}$r = 128$}}%
      \csname LTb\endcsname
      \put(3094,1601){\makebox(0,0)[r]{\strut{}$r = 2$ (stoch)}}%
      \csname LTb\endcsname
      \put(3094,1431){\makebox(0,0)[r]{\strut{}$r = 128$ (stoch)}}%
      \csname LTb\endcsname
      \put(2137,105){\makebox(0,0){\strut{}n}}%
      \csname LTb\endcsname
      \put(2137,2826){\makebox(0,0){\strut{}$S_1$ expected Fisher}}%
    }%
    \gplgaddtomacro\gplbacktext{%
      \csname LTb\endcsname
      \put(3697,822){\makebox(0,0)[r]{\strut{}}}%
      \csname LTb\endcsname
      \put(3697,1175){\makebox(0,0)[r]{\strut{}}}%
      \csname LTb\endcsname
      \put(3697,1527){\makebox(0,0)[r]{\strut{}}}%
      \csname LTb\endcsname
      \put(3697,1879){\makebox(0,0)[r]{\strut{}}}%
      \csname LTb\endcsname
      \put(3697,2232){\makebox(0,0)[r]{\strut{}}}%
      \csname LTb\endcsname
      \put(3872,465){\makebox(0,0){\strut{}$10^{3}$}}%
      \csname LTb\endcsname
      \put(5481,465){\makebox(0,0){\strut{}$10^{4}$}}%
      \csname LTb\endcsname
      \put(6706,822){\makebox(0,0)[l]{\strut{}$10^{-15}$}}%
      \csname LTb\endcsname
      \put(6706,1175){\makebox(0,0)[l]{\strut{}$10^{-12}$}}%
      \csname LTb\endcsname
      \put(6706,1527){\makebox(0,0)[l]{\strut{}$10^{-9}$}}%
      \csname LTb\endcsname
      \put(6706,1879){\makebox(0,0)[l]{\strut{}$10^{-6}$}}%
      \csname LTb\endcsname
      \put(6706,2232){\makebox(0,0)[l]{\strut{}$10^{-3}$}}%
    }%
    \gplgaddtomacro\gplfronttext{%
      \csname LTb\endcsname
      \put(5202,105){\makebox(0,0){\strut{}n}}%
      \csname LTb\endcsname
      \put(5202,2826){\makebox(0,0){\strut{}$S_2$ (split at 0) expected Fisher}}%
    }%
    \gplbacktext
    \put(0,0){\includegraphics[width={368.00bp},height={283.00bp}]{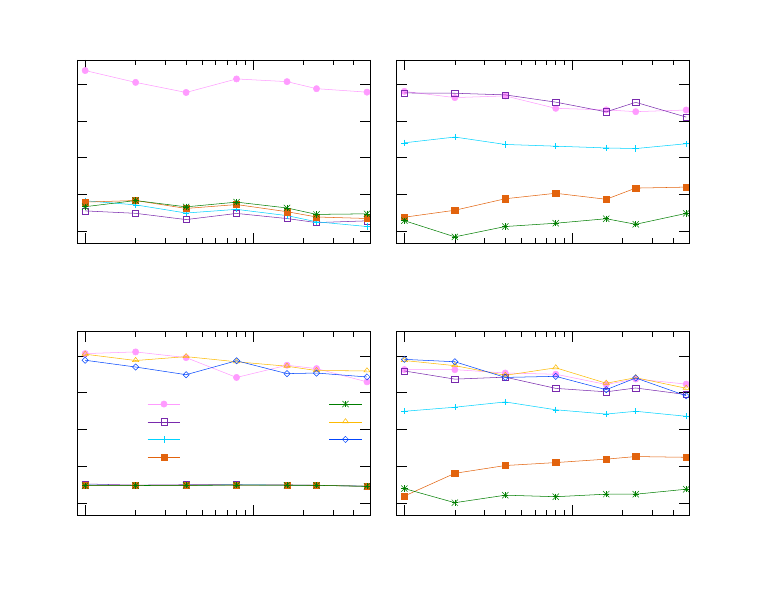}}%
    \gplfronttext
  \end{picture}%
\endgroup

  \vspace{-0.1in}

  \caption{Relative error-type metrics for the approximated gradient and both
  exact and stochastically approximated Fisher information matrices for various
  ranks and sizes.}
  \label{fig:gradfish_err}
\end{figure}
Many of the conclusions from examining the relative error in the log-likelihood
above still apply here. As before, for spectral density $S_1$, the matrix $\bS$
is so well-behaved that even a rank $r=4$ approximation to $\Fr \bS \Fr^H -
\bm{D}$ and its derivatives gives $12+$ digits of accuracy in the gradient and
exact expected Fisher matrix. For the likelihood of $S_2$, due to the roughness
at the origin the rank of the Whittle correction again needs to be increased.
But as one would hope, since a rank of $r=128$ gave effectively every digit in
the log-likelihood, it also gives effectively every digit in the gradient and
exact expected Fisher matrix. Considering that the derivatives of $S_1$ and
$S_2$ with respect to $\theta_2$ are not strictly non-negative, this also
provides a verification that the structure this method exploits also works
well for symmetric Toeplitz matrices that are not positive-definite.

For the stochastic expected Fisher matrix, every trace approximation was done
using $L=72$ SAA vectors, and the SAA-induced error is clearly the dominant
source of disagreement with the true expected Fisher matrix, even for $r=2$.
While it is tempting conclude that one could simply set $r=2$ and obtain a good
estimator for the expected Fisher information matrix in under a second even for
$n = 50\,000$ data points, which may well be true in plenty of circumstances,
the situation can be slightly more complicated than that. When using exact
matrices $\bS$, $\bS_j$, and $\bS_k$, the symmetrized stochastic estimator is an
unbiased estimator. And while the SAA-induced variability in the approximated
expected Fisher using the \emph{matrix} approximation here may be much larger
than the bias induced by setting the rank $r$ too small, the estimator now
having some amount of bias may be a potential source of issues, so we advise
practitioners to be cautious when reducing $r$ from what they deemed necessary
for an accurate log-likelihood.

Finally, as the most direct validation of the method, we conduct a simulation
study and compare the true MLE with the estimator obtained from optimizing our
approximation $\tilde{\ell}(\bth \sv \Fr \by)$ as well as the Whittle and
debiased Whittle approximations. Using the spectral density $S_2$ with true
parameters $\theta_1 = 5$ and $\theta_2 = 35$, we simulate $50$ trials with
$n=10\,000$ data points and compute each estimator as well as the true MLE,
using a rank of $r=128$ in our method for all trials. Figure
\ref{fig:mle_compare} summarizes these results and motivates several
observations.
\begin{figure}[!ht]
  \centering
\begingroup
  \makeatletter
  \providecommand\color[2][]{%
    \GenericError{(gnuplot) \space\space\space\@spaces}{%
      Package color not loaded in conjunction with
      terminal option `colourtext'%
    }{See the gnuplot documentation for explanation.%
    }{Either use 'blacktext' in gnuplot or load the package
      color.sty in LaTeX.}%
    \renewcommand\color[2][]{}%
  }%
  \providecommand\includegraphics[2][]{%
    \GenericError{(gnuplot) \space\space\space\@spaces}{%
      Package graphicx or graphics not loaded%
    }{See the gnuplot documentation for explanation.%
    }{The gnuplot epslatex terminal needs graphicx.sty or graphics.sty.}%
    \renewcommand\includegraphics[2][]{}%
  }%
  \providecommand\rotatebox[2]{#2}%
  \@ifundefined{ifGPcolor}{%
    \newif\ifGPcolor
    \GPcolortrue
  }{}%
  \@ifundefined{ifGPblacktext}{%
    \newif\ifGPblacktext
    \GPblacktexttrue
  }{}%
  \let\gplgaddtomacro\g@addto@macro
  \gdef\gplbacktext{}%
  \gdef\gplfronttext{}%
  \makeatother
  \ifGPblacktext
    \def\colorrgb#1{}%
    \def\colorgray#1{}%
  \else
    \ifGPcolor
      \def\colorrgb#1{\color[rgb]{#1}}%
      \def\colorgray#1{\color[gray]{#1}}%
      \expandafter\def\csname LTw\endcsname{\color{white}}%
      \expandafter\def\csname LTb\endcsname{\color{black}}%
      \expandafter\def\csname LTa\endcsname{\color{black}}%
      \expandafter\def\csname LT0\endcsname{\color[rgb]{1,0,0}}%
      \expandafter\def\csname LT1\endcsname{\color[rgb]{0,1,0}}%
      \expandafter\def\csname LT2\endcsname{\color[rgb]{0,0,1}}%
      \expandafter\def\csname LT3\endcsname{\color[rgb]{1,0,1}}%
      \expandafter\def\csname LT4\endcsname{\color[rgb]{0,1,1}}%
      \expandafter\def\csname LT5\endcsname{\color[rgb]{1,1,0}}%
      \expandafter\def\csname LT6\endcsname{\color[rgb]{0,0,0}}%
      \expandafter\def\csname LT7\endcsname{\color[rgb]{1,0.3,0}}%
      \expandafter\def\csname LT8\endcsname{\color[rgb]{0.5,0.5,0.5}}%
    \else
      \def\colorrgb#1{\color{black}}%
      \def\colorgray#1{\color[gray]{#1}}%
      \expandafter\def\csname LTw\endcsname{\color{white}}%
      \expandafter\def\csname LTb\endcsname{\color{black}}%
      \expandafter\def\csname LTa\endcsname{\color{black}}%
      \expandafter\def\csname LT0\endcsname{\color{black}}%
      \expandafter\def\csname LT1\endcsname{\color{black}}%
      \expandafter\def\csname LT2\endcsname{\color{black}}%
      \expandafter\def\csname LT3\endcsname{\color{black}}%
      \expandafter\def\csname LT4\endcsname{\color{black}}%
      \expandafter\def\csname LT5\endcsname{\color{black}}%
      \expandafter\def\csname LT6\endcsname{\color{black}}%
      \expandafter\def\csname LT7\endcsname{\color{black}}%
      \expandafter\def\csname LT8\endcsname{\color{black}}%
    \fi
  \fi
    \setlength{\unitlength}{0.0500bp}%
    \ifx\gptboxheight\undefined%
      \newlength{\gptboxheight}%
      \newlength{\gptboxwidth}%
      \newsavebox{\gptboxtext}%
    \fi%
    \setlength{\fboxrule}{0.5pt}%
    \setlength{\fboxsep}{1pt}%
    \definecolor{tbcol}{rgb}{1,1,1}%
\begin{picture}(7360.00,5660.00)%
    \gplgaddtomacro\gplbacktext{%
      \csname LTb\endcsname
      \put(633,3644){\makebox(0,0)[r]{\strut{}$2$}}%
      \csname LTb\endcsname
      \put(633,3882){\makebox(0,0)[r]{\strut{}$3$}}%
      \csname LTb\endcsname
      \put(633,4121){\makebox(0,0)[r]{\strut{}$4$}}%
      \csname LTb\endcsname
      \put(633,4360){\makebox(0,0)[r]{\strut{}$5$}}%
      \csname LTb\endcsname
      \put(633,4598){\makebox(0,0)[r]{\strut{}$6$}}%
      \csname LTb\endcsname
      \put(633,4837){\makebox(0,0)[r]{\strut{}$7$}}%
      \csname LTb\endcsname
      \put(633,5075){\makebox(0,0)[r]{\strut{}$8$}}%
      \csname LTb\endcsname
      \put(865,3285){\makebox(0,0){\strut{}$2$}}%
      \csname LTb\endcsname
      \put(1129,3285){\makebox(0,0){\strut{}$3$}}%
      \csname LTb\endcsname
      \put(1392,3285){\makebox(0,0){\strut{}$4$}}%
      \csname LTb\endcsname
      \put(1656,3285){\makebox(0,0){\strut{}$5$}}%
      \csname LTb\endcsname
      \put(1919,3285){\makebox(0,0){\strut{}$6$}}%
      \csname LTb\endcsname
      \put(2183,3285){\makebox(0,0){\strut{}$7$}}%
      \csname LTb\endcsname
      \put(2446,3285){\makebox(0,0){\strut{}$8$}}%
    }%
    \gplgaddtomacro\gplfronttext{%
      \csname LTb\endcsname
      \put(337,4300){\rotatebox{-270.00}{\makebox(0,0){\strut{}$\hat{\theta}_1$}}}%
      \csname LTb\endcsname
      \put(1590,2925){\makebox(0,0){\strut{}$\hat{\theta}_1^{\tiny \textrm{MLE}}$}}%
      \csname LTb\endcsname
      \put(1590,5435){\makebox(0,0){\strut{}Whittle}}%
    }%
    \gplgaddtomacro\gplbacktext{%
      \csname LTb\endcsname
      \put(2712,3644){\makebox(0,0)[r]{\strut{}}}%
      \csname LTb\endcsname
      \put(2712,3882){\makebox(0,0)[r]{\strut{}}}%
      \csname LTb\endcsname
      \put(2712,4121){\makebox(0,0)[r]{\strut{}}}%
      \csname LTb\endcsname
      \put(2712,4360){\makebox(0,0)[r]{\strut{}}}%
      \csname LTb\endcsname
      \put(2712,4598){\makebox(0,0)[r]{\strut{}}}%
      \csname LTb\endcsname
      \put(2712,4837){\makebox(0,0)[r]{\strut{}}}%
      \csname LTb\endcsname
      \put(2712,5075){\makebox(0,0)[r]{\strut{}}}%
      \csname LTb\endcsname
      \put(2945,3285){\makebox(0,0){\strut{}$2$}}%
      \csname LTb\endcsname
      \put(3208,3285){\makebox(0,0){\strut{}$3$}}%
      \csname LTb\endcsname
      \put(3472,3285){\makebox(0,0){\strut{}$4$}}%
      \csname LTb\endcsname
      \put(3735,3285){\makebox(0,0){\strut{}$5$}}%
      \csname LTb\endcsname
      \put(3999,3285){\makebox(0,0){\strut{}$6$}}%
      \csname LTb\endcsname
      \put(4262,3285){\makebox(0,0){\strut{}$7$}}%
      \csname LTb\endcsname
      \put(4526,3285){\makebox(0,0){\strut{}$8$}}%
    }%
    \gplgaddtomacro\gplfronttext{%
      \csname LTb\endcsname
      \put(3669,2925){\makebox(0,0){\strut{}$\hat{\theta}_1^{\tiny \textrm{MLE}}$}}%
      \csname LTb\endcsname
      \put(3669,5435){\makebox(0,0){\strut{}Debiased Whittle}}%
    }%
    \gplgaddtomacro\gplbacktext{%
      \csname LTb\endcsname
      \put(4792,3644){\makebox(0,0)[r]{\strut{}}}%
      \csname LTb\endcsname
      \put(4792,3882){\makebox(0,0)[r]{\strut{}}}%
      \csname LTb\endcsname
      \put(4792,4121){\makebox(0,0)[r]{\strut{}}}%
      \csname LTb\endcsname
      \put(4792,4360){\makebox(0,0)[r]{\strut{}}}%
      \csname LTb\endcsname
      \put(4792,4598){\makebox(0,0)[r]{\strut{}}}%
      \csname LTb\endcsname
      \put(4792,4837){\makebox(0,0)[r]{\strut{}}}%
      \csname LTb\endcsname
      \put(4792,5075){\makebox(0,0)[r]{\strut{}}}%
      \csname LTb\endcsname
      \put(5025,3285){\makebox(0,0){\strut{}$2$}}%
      \csname LTb\endcsname
      \put(5288,3285){\makebox(0,0){\strut{}$3$}}%
      \csname LTb\endcsname
      \put(5552,3285){\makebox(0,0){\strut{}$4$}}%
      \csname LTb\endcsname
      \put(5815,3285){\makebox(0,0){\strut{}$5$}}%
      \csname LTb\endcsname
      \put(6078,3285){\makebox(0,0){\strut{}$6$}}%
      \csname LTb\endcsname
      \put(6342,3285){\makebox(0,0){\strut{}$7$}}%
      \csname LTb\endcsname
      \put(6605,3285){\makebox(0,0){\strut{}$8$}}%
    }%
    \gplgaddtomacro\gplfronttext{%
      \csname LTb\endcsname
      \put(5749,2925){\makebox(0,0){\strut{}$\hat{\theta}_1^{\tiny \textrm{MLE}}$}}%
      \csname LTb\endcsname
      \put(5749,5435){\makebox(0,0){\strut{}Our estimator}}%
    }%
    \gplgaddtomacro\gplbacktext{%
      \csname LTb\endcsname
      \put(633,891){\makebox(0,0)[r]{\strut{}$24$}}%
      \csname LTb\endcsname
      \put(633,1256){\makebox(0,0)[r]{\strut{}$28$}}%
      \csname LTb\endcsname
      \put(633,1621){\makebox(0,0)[r]{\strut{}$32$}}%
      \csname LTb\endcsname
      \put(633,1986){\makebox(0,0)[r]{\strut{}$36$}}%
      \csname LTb\endcsname
      \put(633,2351){\makebox(0,0)[r]{\strut{}$40$}}%
      \csname LTb\endcsname
      \put(784,606){\makebox(0,0){\strut{}$24$}}%
      \csname LTb\endcsname
      \put(1187,606){\makebox(0,0){\strut{}$28$}}%
      \csname LTb\endcsname
      \put(1590,606){\makebox(0,0){\strut{}$32$}}%
      \csname LTb\endcsname
      \put(1993,606){\makebox(0,0){\strut{}$36$}}%
      \csname LTb\endcsname
      \put(2396,606){\makebox(0,0){\strut{}$40$}}%
    }%
    \gplgaddtomacro\gplfronttext{%
      \csname LTb\endcsname
      \put(236,1621){\rotatebox{-270.00}{\makebox(0,0){\strut{}$\hat{\theta}_2$}}}%
      \csname LTb\endcsname
      \put(1590,246){\makebox(0,0){\strut{}$\hat{\theta}_2^{\tiny \textrm{MLE}}$}}%
    }%
    \gplgaddtomacro\gplbacktext{%
      \csname LTb\endcsname
      \put(2712,891){\makebox(0,0)[r]{\strut{}}}%
      \csname LTb\endcsname
      \put(2712,1256){\makebox(0,0)[r]{\strut{}}}%
      \csname LTb\endcsname
      \put(2712,1621){\makebox(0,0)[r]{\strut{}}}%
      \csname LTb\endcsname
      \put(2712,1986){\makebox(0,0)[r]{\strut{}}}%
      \csname LTb\endcsname
      \put(2712,2351){\makebox(0,0)[r]{\strut{}}}%
      \csname LTb\endcsname
      \put(2864,606){\makebox(0,0){\strut{}$24$}}%
      \csname LTb\endcsname
      \put(3266,606){\makebox(0,0){\strut{}$28$}}%
      \csname LTb\endcsname
      \put(3669,606){\makebox(0,0){\strut{}$32$}}%
      \csname LTb\endcsname
      \put(4072,606){\makebox(0,0){\strut{}$36$}}%
      \csname LTb\endcsname
      \put(4475,606){\makebox(0,0){\strut{}$40$}}%
    }%
    \gplgaddtomacro\gplfronttext{%
      \csname LTb\endcsname
      \put(3669,246){\makebox(0,0){\strut{}$\hat{\theta}_2^{\tiny \textrm{MLE}}$}}%
    }%
    \gplgaddtomacro\gplbacktext{%
      \csname LTb\endcsname
      \put(4792,891){\makebox(0,0)[r]{\strut{}}}%
      \csname LTb\endcsname
      \put(4792,1256){\makebox(0,0)[r]{\strut{}}}%
      \csname LTb\endcsname
      \put(4792,1621){\makebox(0,0)[r]{\strut{}}}%
      \csname LTb\endcsname
      \put(4792,1986){\makebox(0,0)[r]{\strut{}}}%
      \csname LTb\endcsname
      \put(4792,2351){\makebox(0,0)[r]{\strut{}}}%
      \csname LTb\endcsname
      \put(4943,606){\makebox(0,0){\strut{}$24$}}%
      \csname LTb\endcsname
      \put(5346,606){\makebox(0,0){\strut{}$28$}}%
      \csname LTb\endcsname
      \put(5749,606){\makebox(0,0){\strut{}$32$}}%
      \csname LTb\endcsname
      \put(6152,606){\makebox(0,0){\strut{}$36$}}%
      \csname LTb\endcsname
      \put(6555,606){\makebox(0,0){\strut{}$40$}}%
    }%
    \gplgaddtomacro\gplfronttext{%
      \csname LTb\endcsname
      \put(5749,246){\makebox(0,0){\strut{}$\hat{\theta}_2^{\tiny \textrm{MLE}}$}}%
    }%
    \gplbacktext
    \put(0,0){\includegraphics[width={368.00bp},height={283.00bp}]{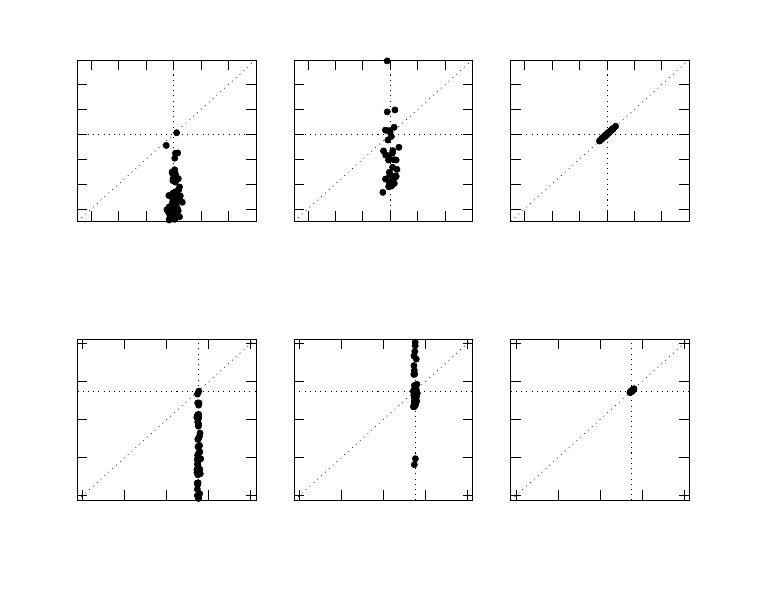}}%
    \gplfronttext
  \end{picture}%
\endgroup

  \vspace{-0.1in}

  \caption{A comparison of the $50$ trials of fitting a process simulated with a
  true SDF of $S_2(\omg)$ with $\theta_1 = 5$ and $\theta_2 = 35$. The first row
  shows a scatter plot of points $(\hat{\theta}_1^{\text{MLE}}, \hat{\theta}_1)$
  with $\hat{\theta}_1$ denoting the alternative estimator, and the lower row
  shows the analog for the estimates of $\theta_2$.}
  \label{fig:mle_compare}
\end{figure}
First, as discussed in the introduction, we readily see that $\hat{\bth}^W$ is
so severely biased that only one trial out of $50$ produced an estimator with
$\hat{\theta}_1 > \theta_{1, \text{true}}$, which obviously makes inferential
questions impossible. The debiased Whittle estimator, on the other hand,
does not show meaningful bias, and comes only at the cost of much higher
variability than the true MLE. Finally, we see that the estimator obtained by
optimizing our log-likelihood is effectively identical to the true MLE in all
cases, with the biggest absolute difference between point estimates being
approximately $10^{-3}$.

\section{Discussion} \label{sec:discussion}

This work introduces a pleasingly simple approximation for the covariance matrix
of the DFT of a stationary time series, which incidentally may provide useful
methods to a much broader range of settings where one encounters symmetric
Toeplitz matrices whose generating sequences have well-behaved Fourier transforms.
In many settings and more general Gaussian process applications, for example
with irregular locations or in multiple dimensions, low rank-type methods are
known to have severe limitations \cite{stein2014}. And for purely approximating
the Toeplitz matrix $\bS$, many of those issues would still be present, so it
is interesting to observe that conjugating with the DFT matrix changes the
algebraic structure of $\bS$ to the degree where a low-rank perturbation of a
diagonal matrix represents $\Fr \bS \Fr^H$ to almost every achievable
significant digit at double precision.

The theory in Section \ref{sec:lowrank_theory} provides a new and complementary
perspective to existing analyses of the accuracy of Whittle approximations.  In
\cite{subba2021}, for example, a theoretical analysis of spectral-domain
approximations is provided where error is characterized by an integer $K$ and
the assumption that the corresponding covariance sequence has a finite $K$-th
moment $\sum_j j^K h_j < \infty$. This work, on the other hand, discusses
non-smooth points of spectral densities and large high-order derivatives as the
principle source of error.  These are of course intimately connected, and the
correspondence of a function's decay rate with the smoothness of its Fourier
transform is well known. If one is designing parametric models via the spectral
density, the considerations that this work motivatives for designing models that
can be approximated well---namely, selecting and designing the number of
non-smooth points, including at endpoints---may be more direct in some circumstances.

\new{As discussed in the introduction,} many methods already exist for
exploiting structure in covariance matrices like $\bS$ and achieving similar or
even slightly better runtime complexity.  What makes this structure
representation interesting, in our opinion, is its simplicity. This
approximation is easy to implement in software, and the resulting $\bm{D} +
\bm{U}\bm{V}^H$ has a very transparent structure.  Unsurprisingly, its runtime
complexity is also very simple and predictable since it is only controlled by
one tuning parameter (the rank $r$). Despite this, it achieves high accuracy for
a very broad class of spectral densities. And for users who are so inclined,
making its rank adaptive would be a very simple modification to the code, for
example following guidelines from \cite{halko2011}. Moreover, error control and
analysis for this structure is more direct than in a general hierarchical matrix
format, which must control the error of individual low rank approximations of
off-diagonal blocks. 

With that said, there are of course circumstances in which the cost of using
this log-likelihood evaluation strategy is not worth the higher prefactor than,
say, the debiased Whittle approximation \cite{sykulski2019}. If one has a
dataset with $50\,000$ measurements that needs to be fitted once, for example,
\new{performing precise estimation with this method takes only one or two
minutes.  And if the point estimates of parameters are themselves scientifically
meaningful, a methodological approach like the one here provides significantly
more precise estimates \emph{and} uncertainty quantification for them, which may
be extremely valuable.} But for an online application, on the other hand, or a
setting in which one must fit a model to tens or hundreds of datasets, a much
cheaper method that runs in $20$ seconds instead of two minutes may be worth the
additional variance of the estimator. \new{With that said, for some families of
spectral densities, like those with many rough points or high-frequency
oscillations, this method is unlikely to provide better estimators than a much
simpler alternative.}

\new{As a final comment, we remind the reader that this method depends crucially
on the fact that $\bS$ is Toeplitz. For irregularly sampled measurements in one or
several dimensions, $\Fr \bS \Fr^H$ does not have such easily extractable
structure. And for gridded data in multiple dimensions, this method could
certainly be extended and applied reasonably directly, but would only
lead to a representation that provides fast matrix-vector products and not a fast
log-determinant. This and other related questions are exciting avenues for
future work.}

\renewcommand{\abstractname}{Acknowledgements}
\begin{abstract}
  The author is grateful to Paul G. Beckman and the two anonymous referees for
  their thoughtful comments, discussion, and proofreading of the manuscript. He
  also thanks Lydia Zoells for her careful copyediting work in the revision
  stage.
\end{abstract}

\bibliography{references_short}

\end{document}